\documentclass[12pt,reqno]{amsart}
\usepackage[margin=1in]{geometry}
\usepackage{amscd,amsfonts,amsmath,amssymb,amsthm,latexsym,mathrsfs,textcomp,verbatim}
\usepackage{accents}
\usepackage{bm}
\usepackage{bbm}
\usepackage{booktabs}
\usepackage{cite}
\usepackage{color}
\usepackage{constants}
\usepackage{csquotes}
\usepackage{enumerate}
\usepackage{etoolbox}
\usepackage{float}
\usepackage{hypbmsec}
\usepackage[dvipsnames]{xcolor}
\usepackage[breaklinks=true,colorlinks=true,linkcolor=black!40!blue,citecolor=black!40!blue,filecolor=black!40!blue ,urlcolor=black!40!blue, pagebackref]{hyperref}
\hypersetup{linktocpage}

\usepackage[capitalise]{cleveref}
\usepackage{mathtools}
\usepackage{soul}
\usepackage{tikz}
\usepackage{tikz-cd}
\usetikzlibrary{shapes.geometric}
\usetikzlibrary{shapes.misc}
\usetikzlibrary{positioning}
\allowdisplaybreaks[4]
\usepackage{hypbmsec}
\swapnumbers
\theoremstyle{plain}
\newtheorem{theorem}[subsubsection]{Theorem}
\newtheorem{corollary}[subsubsection]{Corollary}
\newtheorem{proposition}[subsubsection]{Proposition}
\newtheorem{lemma}[subsubsection]{Lemma}
\theoremstyle{definition}
\newtheorem{definition}[subsubsection]{Definition}

\theoremstyle{remark}
\newtheorem{remark}[subsubsection]{Remark}

\numberwithin{equation}{subsection}

\newcommand{\Z}{\mathbb{Z}}
\renewcommand{\C}{\mathbb{C}}
\newcommand{\R}{\mathbb{R}}
\newcommand{\Q}{\mathbb{Q}}

\renewcommand{\leq}{\leqslant}
\renewcommand{\geq}{\geqslant}

\newcommand{\D}{\mathbb{D}}

\newcommand{\PGL}{\mathrm{PGL}_2}

\newcommand{\BA}{\mathrm{BA}}

\newcommand{\Res}{\mathrm{Res}}

\newcommand{\Spec}{\mathrm{Spec}}

\newcommand{\A}{\mathbb{A}}

\renewcommand{\tilde}{\widetilde}

\renewcommand{\bar}{\overline}

\newcommand{\Hp}{\mathbb{H}}

\renewcommand{\Im}{\mathrm{Im}}

\renewcommand{\pmod}[1]{\, (\mathrm{mod} {\, #1})}

\renewcommand{\Pr}{\mathbb{P}}

\renewcommand{\Re}{\mathrm{Re}}

\newcommand{\tors}{\mathrm{tors}}

\newcommand{\cV}{\mathcal{V}}

\newcommand{\Isog}{\mathrm{Isog}}
\patchcmd{\section}{\scshape}{\bfseries}{}{}
\makeatletter
\renewcommand{\@secnumfont}{\bfseries}
\makeatother

\makeatletter\newcommand{\tpmod}[1]{{\@displayfalse\pmod{#1}}}
\makeatletter
\@namedef{subjclassname@2020}{\textup{2020} Mathematics Subject Classification}
\makeatother

\usepackage{color}

\begin{document}

\title[Real GT for uniformization maps of algebraic Riemann surfaces]{Real geometric transcendence for uniformization maps of algebraic Riemann surfaces}


\author{Arshay Sheth}
\address[Arshay Sheth]{School of Mathematics, Tata Institute of Fundamental Research, Homi Bhabha Road, Colaba, 
Mumbai - 400005, India}

\email{asheth@math.tifr.res.in}
\urladdr{\href{https://sites.google.com/view/arshaysheth/home}{https://sites.google.com/view/arshaysheth/home}}

\author{Matteo Tamiozzo}
\address[Matteo Tamiozzo]{ Universit\'e Sorbonne Paris Nord, Institut Galil\'ee, LAGA, Villetaneuse, 93430, France}
\email{tamiozzo@math.univ-paris13.fr}

\thanks{The first author was supported by funding from the European Research Council under the European Union’s Horizon 2020 research and innovation programme (Grant agreement No. 101001051 — Shimura varieties and the Birch--Swinnerton-Dyer conjecture) during preparation of parts of the article.}


\begin{abstract}
Let $X$ be a smooth connected complex algebraic curve that is not simply connected, and let $\tilde{X}$ be the universal cover of $X$. We study the set of irreducible real algebraic curves in $X$ (seen as a real algebraic surface) containing the image of an arc of a real algebraic curve in $\tilde{X}$. In particular, we give necessary and sufficient conditions on $X$ in order for this set to be non-empty or infinite, and we describe the set explicitly when $X$ is projective of genus one, or $X$ is hyperbolic and its fundamental group is arithmetic.
\end{abstract}

\maketitle
\tableofcontents

\section{Introduction}

\subsection{Motivation, aims and structure of the article}
Let $\exp: \C \rightarrow \C^\times$ be the function sending $z$ to $e^{2 \pi i z}$. By the Kronecker--Weber theorem, every finite abelian extension of $\Q$ is contained in a cyclotomic extension $\Q(\exp(a))$ for some $a \in \Q$. Note that rational numbers are the only algebraic numbers $a$ such that $\exp(a)$ is algebraic, by the Gelfond--Schneider theorem \cite[Corollary 9.3]{rr14}.

Finite abelian extensions of imaginary quadratic fields are contained in those generated by the $j$-invariant of suitable elliptic curves $\C/\Lambda$ with complex multiplication, and by values of the associated Weber function - a multiple of a power of the Weierstrass function $\wp_\Lambda$ \cite[p. 19]{la87} - at preimages of torsion points of $\C/\Lambda$ \cite[Chapter 10]{la87}. Once again, these numbers are characterized by suitable ``bialgebraicity'' properties: on the one hand, algebraic numbers $\tau$ in the upper half-plane $\Hp \subset \C$ such that $j(\tau)$ is algebraic belong to imaginary quadratic fields \cite[Theorem 15.3]{rr14}. On the other hand, let $\Lambda_{\bar{\Q}}\subset \C$ be the $\bar{\Q}$-vector space generated by $\Lambda$. If the invariants $g_2$ and $g_3$ of $\Lambda$ are algebraic then, by \cite[Theorem V]{mas74}, the only numbers $a \in \Lambda_{\bar{\Q}}$ such that $\wp_\Lambda(a)$ is algebraic are preimages of torsion points of $\C/\Lambda$.

\subsubsection{Complex bialgebraic geometry} Geometric counterparts of the above transcendence results have been studied. For instance, the work of Ax \cite{ax71} (see also \cite[Example 4.5]{kuy18}) implies that, if $\hat{Z}\subset \C^n$ is an irreducible algebraic variety that is not contained in a translate of a hyperplane defined over $\Q$, then the functions
\begin{align*}
\hat{Z} & \rightarrow \C \\
(x_1, \ldots, x_n) & \mapsto \exp(x_i), \; \; \; \; 1 \leq i \leq n,
\end{align*}
are algebraically independent over $\C$. Similar functional transcendence results hold true for the $j$-invariant and for Weierstrass elliptic functions, cf. \cite[pp. 1787-1789]{Pil11}, and yield a description of ``bialgebraic subvarieties'' of $\Hp^n$ (uniformizing a product of modular curves) and $\C^n$ (uniformizing a product of elliptic curves) respectively; cf. \cite[\S 4.1]{kuy18}. 
 
As explained in \cite{Pil11, kuy18} and \cite[pp. 64, 65]{zan12}, such a description of bialgebraic varieties plays an important role in the proof of the André--Oort conjecture and analogs thereof.

\subsubsection{Real bialgebraic geometry for algebraic Riemann surfaces}

After identifying $\C$ with $\R^2$, one can investigate versions of the previous results in real algebraic geometry. If $n=1$ - which, in the real setting, is the first interesting case - the relevant results were proved in \cite[Theorem 2.2.4, Theorem 3.4.2]{Tam23} for the exponential function and the $j$-invariant. In the latter case, real bialgebraic curves in $\Hp$ arise from tori in $\mathrm{GL}_{2, \Q}$ that do not give rise to Shimura varieties.

The first goal of this article is to extend the above results to Weierstrass elliptic functions $\wp_\Lambda$ attached to lattices $\Lambda \subset \C$; this is achieved in \S \ref{sec:bialgweier}, cf. Theorem \ref{mainthmwp}. This result can be restated in terms of the uniformization map $p: \R^2 \simeq \C \rightarrow \C/\Lambda$ - seeing the target as a real algebraic surface. Thus, we obtain a description of real (irreducible) plane curves $\hat{C} \subset \A^2_\R$ such that $\hat{C}(\R)$ is infinite and $p(\hat{C}(\R))$ is contained in an algebraic curve $C$ in $\C/\Lambda$, cf. Theorem \ref{mainthmell}.

As an example, let $\hat{C} \subset \A^2_\R$ be an irreducible curve with infinitely many real points. If $\Lambda=\Z[i]$, then the set $p(\hat{C}(\R))$ is contained in a real algebraic curve in $\C/\Z[i]$ if and only if $\hat{C}$ is a line with rational slope. On the other hand, for a very general elliptic curve $\C/\Lambda$, the image $p(\hat{C}(\R))$ is Zariski dense in $\C/\Lambda$ for every $\hat{C}$ as above, cf. \S \ref{subs:modcrit}. For instance, this is the case if $\Lambda=\Z+\Z\tau$ for $\tau=t+i$, where $t \in \R$ is any transcendental number.

The functional transcendence phenomena observed for the uniformization map of projective curves of genus one extend to hyperbolic algebraic curves; this is the subject of \S \ref{sect:hypcurves}. 
All in all, the results in this article, together with those in \cite[\S 2]{Tam23}, allow to obtain uniform statements describing the set of real bialgebraic curves  in an arbitrary (non simply connected) algebraic Riemann surface, as we will explain in \S \ref{ssec:overview} below.

\begin{remark}
In \S \ref{sec-rrayn} we observe that Raynaud's theorem on torsion points in subvarieties of complex abelian varieties descends to an analogous result on subvarieties of real abelian varieties. For the Weil restriction $Y=\mathrm{Res}_{\C/\R}X$ of a complex elliptic curve, we deduce in Theorem \ref{thm:mmrealell} a relation between curves in $Y$ containing infinitely many torsion points and bialgebraic curves in $Y(\R)$, analogous to the one that holds when the ambient space is a complex abelian variety. This is a further example - in addition to those studied in \cite{Tam23}, \cite{Tam26} - suggesting that several statements related to the André--Oort conjecture and variants or generalizations thereof (see \cite{kuy18}, \cite{pil22}) could have analogs in real algebraic geometry.
\end{remark}

\subsection{Overview of the main results}\label{ssec:overview}
Let $X$ be a smooth connected complex algebraic curve such that the Riemann surface $X(\C)$ is not simply connected. Let $p \colon \tilde{X} \rightarrow X(\C)$ be the universal cover of $X(\C)$. We regard $X(\C)$ as the set of real points of the real surface $Y=\Res_{\C/\R}X$. By the uniformization theorem for Riemann surfaces (whose wonderful historical development is summarized in the introduction of \cite{ger10}) the universal cover $\tilde{X}$ is isomorphic either to $\C$ - which we identify with $\R^2$ - or to the upper half-plane $\Hp$. In both cases, we have a notion of real algebraic curve in $\tilde{X}$, and of arc inside such a curve (cf. Definition \ref{def:balgell}(2) and Definition \ref{def-bialgh}(1)). We call a subset $\mathcal{C} \subset Y(\R)$ a \textit{bialgebraic curve} if there is an irreducible curve $C \subset Y$ such that $\mathcal{C}=C(\R)$, and there exists an arc $\mathcal{I}$ in an algebraic curve $\hat{\mathcal{C}} \subset \tilde{X}$ such that $p(\mathcal{I})$ is contained in $\mathcal{C}$.

Every biholomorphism of $X$ lifts to a biholomorphism of $\tilde{X}$, which preserves real algebraic curves in $\tilde{X}$; this implies that bialgebraic curves in $Y(\R)$ do not depend on the chosen covering map $p$, and that the image of a bialgebraic curve $\mathcal{C} \subset Y(\R)$ via an automorphism of $X$ is bialgebraic (see Remark \ref{rem-alginH}(2) if $\tilde{X}\simeq \Hp$; if $\tilde{X}\simeq \C$ the argument is similar). We say that two bialgebraic curves $\mathcal{C}_1$ and $\mathcal{C}_2$ are equivalent if there is an automorphism $a: X \rightarrow X$ such that $a(\mathcal{C}_1)=\mathcal{C}_2$.

\begin{remark}
 Let $\hat{\mathcal{C}}$ be an algebraic curve in $\tilde{X}$, and let $q: X \rightarrow U$ be a finite morphism with values in $\Pr^1_\C$ or an open subset thereof (if $X(\C)$ is not compact). There is an arc $\mathcal{I}\subset \hat{\mathcal{C}}$ such that $p(\mathcal{I})$ is contained in an irreducible curve in $Y$ if and only if there is an arc $\mathcal{I} \subset \hat{\mathcal{C}}$ such that $q \circ p(\mathcal{I})$ is contained in an irreducible curve in $\Res_{\C/\R}U$. Therefore, when $\tilde{X}\simeq \Hp$ the study of bialgebraic curves in $Y(\R)$ can be thought of, in more classical terms, as the study of the functional transcendence properties of Poincaré's fuchsian functions, seen as functions of two real variables. In fact, geodesics in the open unit disc whose image via suitable fuchsian functions is contained in the real axis already appear in Poincaré's work on uniformization of the complex projective line minus a finite set of real points \cite[p. 1199]{poin81} (see \cite[\S 2.3]{beg12} for a modern account).
\end{remark}

\begin{theorem}\label{thm:altogether} Let $X$ be a smooth connected complex algebraic curve such that $X(\C)$ is not simply connected, and let $Y=\Res_{\C/\R}X$. The following assertions hold true.
\begin{enumerate}
    \item There exist bialgebraic curves in $Y(\R)$ if and only if there is a smooth connected complex curve $X'$ defined over the reals with non-empty set of real points, and a finite étale map $X' \rightarrow X$.
    \item The set of equivalence classes of bialgebraic curves in $Y(\R)$ is infinite if and only if $X$ is an elliptic curve with complex multiplication, or a curve with universal cover $\Hp$ whose fundamental group is an arithmetic lattice in $\PGL(\R)^+$.
\end{enumerate}
\end{theorem}

\begin{remark}\leavevmode
\begin{enumerate}
    \item The condition we impose in the definition of bialgebraic curve is essentially the minimal one making the arguments in this article go through; we refer the reader to Remark \ref{rem:defbalg} for some motivation behind the definition, and to \cite[pp. 64, 113]{zan12} for two situations in which a very similar notion naturally appears.
    
    It turns out that, in our situation, a bialgebraic curve in $Y(\R)$ contains the full image of an algebraic curve in $\tilde{X}$: see Remark \ref{rem:defbalg}(2) and \S \ref{ssec: bialisgeo} for $X$ not isomorphic to $\A^1 \smallsetminus \{0\}$; in the remaining case, the claim follows from the arguments in the proof of \cite[Theorem 2.2.4]{Tam23}.
    \item Theorem \ref{thm:altogether} follows from the main results in this article (especially Theorem \ref{mainthmell} and Theorem \ref{thm-main}) and \cite[Theorem 2.2.4]{Tam23}; the details of the proof will be given in \S \ref{ssec:altogether}.
    \item Other finiteness statements for hyperbolic curves with non-arithmetic fundamental group, that inspired part of the above theorem, are mentioned in Remark \ref{rem:finnonar}(2).
\end{enumerate}
\end{remark}

\subsubsection{Explicit descriptions of bialgebraic curves}

Along the way, we obtain several descriptions of bialgebraic curves in $Y(\R)$. 
For arbitrary $X$, all bialgebraic curves $\mathcal{C}=C(\R)$ are obtained via the following construction: take a finite étale cover $X' \rightarrow X$ with $X'=C' \times_\R \C$ as in Theorem \ref{thm:altogether}, inducing a map $q: Y'=\Res_{\C/\R}X' \rightarrow Y$, and let $C=q(C')$. The fact that this construction accounts for all bialgebraic curves in $Y(\R)$ when $X$ is compact of genus one follows from Remark \ref{rem-bialgrealell}, and for $X$ hyperbolic it is the content of Proposition \ref{prop-bialfromrealcover}. Finally, for $X=\A^1_\C \smallsetminus \{0\}$ it can be checked directly using the argument in the proof of \cite[Theorem 2.2.4]{Tam23}. 

Roughly speaking, we see that the functional transcendence properties of the uniformization map $\tilde{X} \rightarrow X$ ``detect'' real finite étale covers of $X$.

For special types of curves, we can be more explicit.
\begin{enumerate}
    \item For an elliptic curve $X$, we describe explicitly the curves in $\R^2$ whose image in $Y(\R)$ is contained in an algebraic curve in Theorem \ref{mainthmwp}, and bialgebraic curves in $Y(\R)$ in \S \ref{ssec:gpbialgell}; the latter are the real points of translates of elliptic curves in $Y$.
    \item Suppose that $X(\C)$ is compact hyperbolic with arithmetic fundamental group, commensurable with the group of norm one units of an order in a quaternion algebra $B$ over a totally real field $F$. In this case, all bialgebraic curves arise from non-CM quadratic extensions of $F$ embedded in $B$, cf. \S \ref{ssec:bialgarit}. This generalizes the description of bialgebraic curves in the modular surface $\mathrm{SL}_2(\Z)\backslash \Hp$ given in \cite[\S 3]{Tam23}.
    \item When $X(\C)$ is compact, bialgebraic curves in $Y(\R)$ are precisely Zariski closures of closed geodesics if and only if $X$ is either an elliptic curve with complex multiplication (cf. Theorem \ref{mainthmell}), or a hyperbolic curve with arithmetic fundamental group (cf. Corollary \ref{cor:cocogeod}).
\end{enumerate}

\subsubsection{Proof strategy}
The basic idea to study bialgebraic curves in $Y(\R)$, as in \cite{Tam23}, is to reduce via a base change argument to the study of bialgebraic curves in a product of two complex curves. The relevant complex inputs in this article are Theorem \ref{thm:pila} and Proposition \ref{mainprop}; we also use Margulis' theorem \cite[Theorem 1']{mar91} to deduce Theorem \ref{thm-klin} from Proposition \ref{mainprop}, hence obtain finiteness of bialgebraic curves in $Y(\R)$ when $X$ is hyperbolic and has non-arithmetic fundamental group. 
In \cite{Tam23}, the real and complex situations were related via direct computations, relying on explicit formulas involving the exponential function, the $j$-invariant or modular polynomials. We use similar explicit methods in \S \ref{sec:bialgweier} but, in the rest of the article - especially in \S \ref{sect:hypcurves} - most computations are replaced by more conceptual arguments. Besides allowing to deal with arbitrary complex algebraic curves, such arguments could be more amenable to generalization to the study of real bialgebraic subsets of higher dimensional locally symmetric spaces.

\subsection{Notation and conventions}
The following notation and conventions will be used throughout the article.
\begin{enumerate}[-]
    \item For an abelian group $G$, we denote by $G_{\tors}$ the subgroup of torsion elements.
    \item Let $V$ be a complex vector space, and let $R\subset \C$ be a subring. For $v_1, v_2 \in V$ we denote by $\langle v_1, v_2 \rangle_R$ the $R$-submodule of $V$ generated by $v_1$ and $v_2$.
    \item Let $\Lambda_1, \Lambda_2$ be two lattices in $\C$. We denote by $\Isog(\Lambda_1, \Lambda_2)$ the set of $\alpha \in \C$ such that $\alpha\Lambda_1 \subset \Lambda_2$; it is an additive subgroup of $\C$, free of rank at most two. We say that $\Lambda_1$ is isogenous to $\Lambda_2$ if $\Isog(\Lambda_1, \Lambda_2)$ is non-zero (equivalently, the complex elliptic curves $\C/\Lambda_1$ and $\C/\Lambda_2$ are isogenous). A CM lattice is a lattice $\Lambda$ such that $\Isog(\Lambda, \Lambda)$ has rank two.

    The image of a lattice $\Lambda \subset \C$ via complex conjugation is denoted by $\overline{\Lambda}$.
    \item By a subvariety of an algebraic variety we mean a closed (possibly reducible) reduced subscheme. We often tacitly identify complex algebraic varieties with their complex points. We identify $\R^2$ with $\C$ sending $(x, y)$ to $x+iy$.
\end{enumerate}

\subsection*{Acknowledgements} We are grateful to Lambert A'Campo, Bruno Klingler, Martí Roset Julià and Olivier Wittenberg for useful exchanges and discussions. In particular, we thank Klingler for explaining the strategy of proof of Theorem \ref{thm-klin}, which led to much of the content of \S \ref{sect:hypcurves}. We thank A'Campo for suggesting the statement of Corollary \ref{cor:cocogeod}, and Roset Julià for pointing out the reference \cite{gmx21}. We thank Joseph Harrison, Nicola Ottolini and Harry Schmidt for their interest in this work and for helpful discussions.

\section{Real bialgebraic curves for Weierstrass elliptic functions}\label{sec:bialgweier}

Let $\Lambda\subset \C$ be a lattice. In this section we describe real plane algebraic curves whose image via the Weierstrass $\wp$-function, seen as a function from $\R^2\smallsetminus \Lambda$ to $\R^2$, is contained in a real algebraic curve.

\subsection{Statement of the main theorem} Fix a lattice $\Lambda \subset \C\simeq \R^2$ with attached Weierstrass $\wp$-function $\wp_\Lambda: \C \rightarrow \Pr^1(\C)$. Consider the function
\begin{align*}
    \mathcal{P}_{\Lambda}: \R^2 \smallsetminus \Lambda & \rightarrow \R^2\\
    (x, y) & \mapsto (\Re (\wp_\Lambda(x+iy)), \Im (\wp_\Lambda(x+iy))).
\end{align*}

\begin{definition}\label{def-bialg}
A non-empty subset $\cV \subsetneq \R^2$ is called weakly bialgebraic for $\mathcal{P}_{\Lambda}$ if it satisfies the following properties:
\begin{enumerate}
\item there exist algebraic subvarieties $V \subset \A^2_\R$ and $W \subsetneq \A^2_\R$ such that $\cV=V(\R)$ and $\mathcal{P}_\Lambda(\cV \cap (\R^2 \smallsetminus \Lambda)) \subset W(\R)$;
\item the set $\cV$ cannot be written in the form $V_1(\R) \cup V_2(\R)$, where $V_1, V_2 \subset \A^2_\R$ are algebraic subvarieties and the inclusions $V_i(\R)\subset \cV$ are proper for $i=1, 2$. 
\end{enumerate}
\end{definition}

\begin{remark} \leavevmode
\begin{enumerate}
    \item The above definition mirrors \cite[Definition 2.2.2, Definition 3.1.2]{Tam23}. The requirement that $\cV$ is non-empty was mistakenly not included in \emph{loc. cit.}; it is needed in order for the main results stated in \cite{Tam23} to be true.
    \vspace{1mm}
    \item  In what follows, we will mostly work with a fixed lattice $\Lambda$; furthermore, we will never consider the stronger notion of ``strongly bialgebraic'' sets \cite[Definition 2.2.2]{Tam23}. Therefore, weakly bialgebraic sets for $\mathcal{P}_\Lambda$ will often be called just ``bialgebraic sets'', and $\mathcal{P}_{\Lambda}$ will often be denoted by $\mathcal{P}$.
\end{enumerate}
\end{remark}

The main result of this section is the following.

\begin{theorem} \label{mainthmwp}
Let $\Lambda$ be a lattice in $\C$. The following assertions hold true.

\begin{enumerate}

\item If $\Lambda$ is not isogenous to $\overline{\Lambda}$, then the only weakly bialgebraic sets for $\mathcal{P}_\Lambda$ are singletons. 

\item Suppose that $\Lambda$ is not a CM lattice and that $\Lambda$ is isogenous to $\overline{\Lambda}$. Let $\gamma \in \C$ be such that $\Isog(\Lambda, \overline{\Lambda})=\mathbb Z \cdot \gamma$. 
\begin{enumerate}
    \item If $|\gamma| \not \in \Q$, then the only weakly bialgebraic sets for $\mathcal{P}_\Lambda$ are singletons.
    \item If $|\gamma| \in \Q$, then weakly bialgebraic sets for $\mathcal{P}_\Lambda$ are singletons and translates of the lines
    \begin{equation*}
        \mathcal{L}_1=\{(x, y) \in \R^2 \mid \sqrt{\gamma}(x+iy) \in \R\} \text{ and } \mathcal{L}_2=\{(x, y) \in \R^2 \mid \sqrt{\gamma}(x+iy) \in i\R\}.
    \end{equation*}
\end{enumerate}

\item If $\Lambda$ is a CM lattice, then weakly bialgebraic sets for $\mathcal{P}_\Lambda$ are singletons and translates of the lines through the origin in $\R^2$ containing a non-zero point of $\Lambda$.
\end{enumerate}
\end{theorem}

\subsection{An Ax--Lindemann--Weierstrass type theorem} Let $\Lambda$ be a lattice in $\C$. Our proof of the above theorem rests on a functional transcendence result for the map
\begin{equation*}
    \wp_{\Lambda \times \bar{\Lambda}}=\wp_{\Lambda}\times \wp_{\overline{\Lambda}}: \C^2 \rightarrow \Pr^1(\C)\times \Pr^1(\C).
\end{equation*}

\begin{definition}
An irreducible algebraic curve $\hat{Z}\subset \C^2$ is bialgebraic for $\wp_{\Lambda \times \bar{\Lambda}}$ if $\wp_{\Lambda \times \bar{\Lambda}}(\hat{Z})$ is not Zariski dense in $\Pr^1(\C)\times \Pr^1(\C)$.
\end{definition}

\begin{theorem}[Brownawell--Kubota\cite{brodaku77}, Pila \cite{Pil11}]\label{thm:pila}
If $\hat{Z} \subset \C^2$ is a bialgebraic curve for $\wp_{\Lambda \times \bar{\Lambda}}$, then $\hat{Z} = L+\sigma$ for some $\sigma \in \C^2$ and some one-dimensional $\C$-vector subspace $L$ of $\C^2$ such that $L \cap (\Lambda \times \bar{\Lambda})$ generates $L$ as an $\R$-vector space. Conversely, every $\hat{Z} = L+\sigma$ as above is bialgebraic for $\wp_{\Lambda \times \bar{\Lambda}}$.
\end{theorem}
\begin{proof}
    If $\hat{Z}$ is a line of the form $L + \sigma$ as in the statement then the image of $\hat{Z}$ in $\C/\Lambda \times \C/\bar{\Lambda}$ is a translate of a one-dimensional complex torus; hence, the image of $\hat{Z}$ in $\Pr^1(\C)\times \Pr^1(\C)$ is an algebraic curve. Conversely, a bialgebraic curve $\hat{Z}$ for $\wp_{\Lambda \times \bar{\Lambda}}$ has image in $\C/\Lambda \times \C/\bar{\Lambda}$ that is not Zariski dense. Therefore, \cite[Theorem 9.2]{Pil11} implies that the coordinate functions on $\hat{Z}$ cannot be geodesically independent, in the sense of \cite[Definition 1.5(2)]{Pil11}. It follows that $\hat{Z}$ is contained in, hence equal to, a line $L+\sigma$ as in the statement.
\end{proof}
\begin{remark}
Let us mention two alternative approaches to the proof of the first assertion in the above theorem. Let $\hat{Z}\subset \C^2$ be an algebraic curve.
\begin{enumerate}
    \item One can use \cite[Corollary 2]{brodaku77}, which implies that, for an open neighborhood $U\subset \hat{Z}$ of a smooth point biholomorphic to an open in $\C$ and where $\wp_{\Lambda \times \overline{\Lambda}}$ has no poles, if $\wp_{\Lambda \times \overline{\Lambda}}(U)$ is contained in an algebraic curve then $U$ is contained in a line.
    \item The assertion follows from the fact that $\hat{Z}$ is stabilized by a non-zero element of $\Lambda \times \overline{\Lambda}$. This fact can be shown using the argument in \cite[Step (v), p. 64]{zan12}, or adapting the discussion in Lemma \ref{lem:gammastabhatz} and \S \ref{sec:gammastabhatz} below.
\end{enumerate}
\end{remark}

\subsection{Relating real and complex bialgabraic sets}\label{realvscompbalg}

Fix a lattice $\Lambda \subset \C$.
\subsubsection{Preliminaries}\label{subsect:prelim}
Let $\cV=V(\R) \subset \R^2 $ be a bialgebraic set that is not a singleton. We can write $\cV$ as the zero locus of a single polynomial $R \in \R[X, Y]$ (taking the sum of the squares of a set of generators of the ideal defining $V$). By point (2) of Definition \ref{def-bialg}, the set $\cV$ must be the vanishing locus of an irreducible factor $P \in \R[X, Y]$ of $R$, which is also irreducible in $\C[X, Y]$ because $\cV$ is infinite  \cite[\href{https://stacks.math.columbia.edu/tag/0G69}{Tag 0G69}]{stacks-project}.

\subsubsection{The base change diagram}
Consider the maps
\begin{equation*}
\begin{aligned}[t]
  f\colon & \C^2  \rightarrow \C^2 \\ & (v, w) \mapsto (v+iw, v-iw)
\end{aligned}
\quad \hspace{5mm} \quad
\begin{aligned}[t]
  g \colon & \C^2  \rightarrow \C^2 \\
 & (a, b) \mapsto \left ( \frac{a+b}{2}, \frac{a-b}{2i} \right ).
\end{aligned}
\end{equation*}
As $\overline{ \wp_\Lambda(z) }= \wp_{\overline \Lambda}(\overline z)$ for every $z \in \C \smallsetminus \Lambda$, we have the following commutative diagram:

\begin{center}
\begin{equation} \label{commdiagram}
\begin{tikzcd}
\R^2 \smallsetminus \Lambda \arrow[r, "\mathcal{P}"] \arrow[hookrightarrow, d] &[0.5em] \R^2 \arrow[hookrightarrow, r] & \C^2 \arrow[d, "\iota \circ g^{-1}"]\\
\C^2 \arrow[r, "f"]  & \mathbb C^2 \arrow[r, "\wp_{\Lambda \times \bar{\Lambda}}"] & \mathbb \Pr^1(\C)^2. 
\end{tikzcd}
\end{equation}
\end{center}

The top right horizontal map (resp. the left vertical map) above is (resp. is induced by) the inclusion $\R^2 \subset \C^2$, and the map $\iota: \C^2 \rightarrow \Pr^1(\C)^2$ is given by the inclusion $\C \subset \Pr^1(\C)$ on each factor. 

\begin{lemma}\label{lem:realcompbalg}
    Let $\hat{Z}_P\subset \C^2$ be the complex curve with equation $P=0$. The curve $f(\hat{Z}_P)$ is bialgebraic for $\wp_{\Lambda \times \bar{\Lambda}}$.
\end{lemma}

\begin{proof}
Let $\tilde{\cV}=\cV \cap (\R^2 \smallsetminus \Lambda)$. As $\cV$ is bialgebraic, there exists a non-constant polynomial $Q \in \R[X, Y]$ such that $\mathcal{P}(\tilde{\cV})$ is contained in the real algebraic curve with equation $Q=0$. Let $Z_Q^\circ$ be the complex plane curve with equation $Q=0$, and let $Z_Q\subset \Pr^1(\C)^2$ be the Zariski closure of $\iota \circ g^{-1}(Z_Q^\circ)$.

Let $q=\wp_{\Lambda \times \bar{\Lambda}}\circ f$. The commutativity of the diagram \eqref{commdiagram} implies that $A=\hat{Z}_P \cap q^{-1}(Z_Q)$ contains $\tilde{\cV}$. Note that, as $\cV$ is infinite, it contains a smooth point of $\hat{Z}_P$, hence $\tilde{\cV}$ contains an open subset homeomorphic to an interval. It follows that $A$ contains an open subset of $\hat{Z}_P$ biholomorphic to a disc. On the other hand, since $P$ is irreducible, the curve $\hat{Z}_P$ satisfies the equivalent conditions of \cite[Theorem, p. 168]{grre84}. As the analytic subset $A\subset \hat{Z}_P$ is not thin (in the sense of \cite[p. 132]{grre84}), we must have $A=\hat{Z}_P$, therefore $q(\hat{Z}_P)\subset Z_Q$.
\end{proof}

The above lemma tells us that a bialgebraic set for $\mathcal{P}_\Lambda$ that is not a singleton gives rise to a curve in $\C^2$ bialgebraic for $\wp_{\Lambda \times \bar{\Lambda}}$. In the other direction, we have the following result.

\begin{lemma} \label{check}
Let $\cV \subset \R^2$ be a non-empty subset of the form $\cV=V(\R)$ for some algebraic subvariety $V \subset \mathbb{A}^2_\R$, satisfying condition (2) in Definition \ref{def-bialg}. If $f(\cV)$ is contained in a bialgebraic curve for $\wp_{\Lambda \times \overline{\Lambda}}$ then $\cV$ is weakly bialgebraic for $\mathcal{P}_\Lambda$. 
\end{lemma}

\begin{proof}
By the commutativity of the diagram \eqref{commdiagram}, the image of $\mathcal{P}(\cV \cap (\R^2 \smallsetminus \Lambda))$ in $ \C^2$ is contained in an algebraic curve in $\C^2$, vanishing locus of a non-constant polynomial $R \in \C[X, Y]$. Therefore, the set $\mathcal{P}(\cV \cap (\R^2 \smallsetminus \Lambda))$ is contained in the real vanishing locus of $R \overline{R}  \in \R[X, Y]$; hence, the set $\cV$ is bialgebraic.
\end{proof}

\subsection{Linear description of weakly bialgebraic subsets for $\mathcal{P}_\Lambda$}\label{ssec:linbalg}
The notation of \S \ref{subsect:prelim} is in force.

\subsubsection{}\label{subsec:descW} By Theorem \ref{thm:pila} and Lemma \ref{lem:realcompbalg}, we can write $f(\hat{Z}_P)=L+\sigma$ for some $\sigma \in \C^2$ and some one-dimensional $\C$-vector subspace $L$ of $\C ^2$ of the form  $\langle w_1, w_2 \rangle_\R$, with $w_1$ and $w_2$ belonging to $\Lambda \times \overline{\Lambda}$. In particular, we have 
\begin{equation} \label{basechangeresult2}
f(\cV) \subset L+\sigma. 
\end{equation}

Let us write $w_1=(w_{11}, w_{12}) \in \Lambda \times \overline{\Lambda} $ and $w_2=(w_{21}, w_{22}) \in \Lambda \times \overline{\Lambda}$. Note that, as $\cV$ is not a singleton, the subspace $L$ cannot be equal to $\{0\} \times \C$. Thus, we may write
\begin{equation*}
    L=\{(x, y) \in \C^2 \mid y=rx\}
\end{equation*}
for some $r \in \C$, and we find
\begin{equation} \label{double}
w_{12}=r w_{11} \text{ and } w_{22}= r w_{21}. 
\end{equation}
We also note for future use that, since $L$ is a one-dimensional $\mathbb C$-vector space generated by $w_1$ and $w_2$ as a real vector space, there exists $t \in \mathbb C \smallsetminus \mathbb R$ such that $w_1= t w_2$, hence 
\begin{equation} \label{eqnforw11}
w_{11}=t w_{21} \text{ and } w_{12}= t w_{22}. 
\end{equation}

\subsubsection{} \label{sectionr}
Take two distinct elements $(x, y)$ and $(x', y')$ in $\cV$, and let $(x'', y'')=(x-x', y-y')$. By \eqref{basechangeresult2} and the definition of $f$ we have 
\begin{equation*} \label{eqnforr}
x''-iy''= r (x''+iy''), \text{ hence } |r|=1.
\end{equation*}

Write, thanks to Hilbert's Theorem 90, $r=\frac{\rho}{\bar{\rho}}$ with $\rho \in \C^\times$; we see that $\cV$ is contained in - and hence equal to - a translate of the line $\rho^{-1}\R$.

\begin{lemma} \label{translates}
Let $\mathcal{L} \subset \mathbb R^2$ be a line. If $\mathcal{L}$ is weakly bialgebraic for $\mathcal{P}_\Lambda$, then all its translates are weakly bialgebraic for $\mathcal{P}_\Lambda$ as well.
\end{lemma}

\begin{proof}
We have $f(\mathcal{L}) \subset L+\sigma$ for some $L$ and $\sigma$ as in \eqref{basechangeresult2}. Hence, for every $(c_1, c_2) \in \R^2$, the set $f(\mathcal{L}+(c_1, c_2))$ is contained in $L+\sigma+ (c_1+ ic_2, c_1-ic_2)$. The latter is bialgebraic for $\wp_{\Lambda \times \overline{\Lambda}}$ by Theorem \ref{thm:pila}; thus, the set $\mathcal{L}+(c_1, c_2)$ is weakly bialgebraic for $\mathcal{P}_\Lambda$ by Lemma \ref{check}. 
\end{proof}

\subsection{Proof of Theorem \ref{mainthmwp} when $\Lambda$ is not isogenous to $\overline{\Lambda}$}\label{subsec:lamdanotisogbar}

Let us prove Theorem \ref{mainthmwp}(1). Assume that there is a bialgebraic subset $\cV \subset \R^2$ that is not a singleton; we use the notation of \S \ref{subsect:prelim} and \S \ref{subsec:descW}. Let $\Lambda_1=\langle w_{11}, w_{21} \rangle_{\Z}$ and $\Lambda_2=\langle w_{12}, w_{22} \rangle_{\Z}$. We claim that $w_{11}$ and $w_{21}$ are linearly independent over $\R$. Indeed, by \eqref{double} the numbers $w_{11}$ and $w_{21}$ are non-zero. If $c \in \mathbb R$ is such that $w_{21}= c w_{11}$ then \eqref{eqnforw11} yields $w_{11}=c t w_{11}$, hence $c t=1$; this contradicts the fact that $t$ does not belong to $\R$. 

It follows that $\Lambda_1$ is a full sublattice of $\Lambda$; since $r \Lambda_1= \Lambda_2$, the lattice $\Lambda$ is isogenous to $\overline{\Lambda}$.

\subsection{Proof of Theorem \ref{mainthmwp} when $\Lambda$ is isogenous to $\overline{\Lambda}$ and not CM}\label{sec:proofnoncm}

Assume that $\Lambda$ is not CM and that $\Lambda$ is isogenous to $\overline{\Lambda}$. Choose $\gamma \in \C$ such that $\Isog(\Lambda, \overline{\Lambda})= \Z \cdot \gamma$. Note that $|\gamma|^2 \Lambda \subset \Lambda$, therefore $|\gamma|^2 \in \Z$. Let $\cV \subset \R^2$ be a bialgebraic subset that is not a singleton, and let $r, L$ be as in \S \ref{subsec:descW}.

\begin{lemma} \label{rvalue}
In the above situation, we have $\gamma r^{-1} \in \mathbb Q$.
\end{lemma}

\begin{proof}
Let $\Lambda'=\langle w_{12},  \gamma w_{11} \rangle_\mathbb Z \subset \overline{\Lambda}$. If $\gamma r^{-1} \not \in \Q$ then \eqref{double} implies that $\Lambda'=\langle rw_{11},  \gamma w_{11} \rangle_\mathbb Z$ is a free $\mathbb Z$-module of rank two, hence there exists an integer $N \geq 1$ such that $N \overline{\Lambda} \subset \Lambda'$. On the other hand, equation \eqref{eqnforw11} implies that $\Lambda' \subset t \overline{\Lambda}$, therefore we obtain $\frac{N}{t}\overline{\Lambda}\subset \overline{\Lambda}$. But $\overline{\Lambda}$ is not a CM lattice, hence we must have $Nt^{-1} \in \mathbb Z$, contradicting the fact that $t \not \in \mathbb R$.
\end{proof}

\subsubsection{Proof of Theorem \ref{mainthmwp}(2)}\label{subsec: proof(2)}

With the notation fixed at the beginning of \S \ref{sec:proofnoncm}, by Lemma \ref{rvalue} and the fact that $|r|=1$ (cf. \S \ref{sectionr}) we have $|\gamma| \in \Q$. This proves point $(a)$. Let us now suppose that $a=|\gamma| \in \Q$, and let us prove point $(b)$. As observed above, we also have $|\gamma|^2 \in \Z$, hence $|\gamma| \in \Z$. Furthermore, note that
\begin{equation}\label{rvsgamma}
    r=\pm \frac{\gamma}{|\gamma|}.
\end{equation}
If in addition $\gamma$ belongs to $\Q$ then \eqref{rvsgamma} implies that $r=\pm1$, and the discussion in \S \ref{sectionr} shows that $\cV$ is a translate of the $x$-axis or of the $y$-axis; by Lemma \ref{translates}, it remains to show that the two coordinate axes are bialgebraic. The image of the $x$-axis (resp. of the $y$-axis) via $f$ is contained in the line $L_1=\{ (x, y) \in \C ^2 \mid x=y\}$ (resp. $L_2= \{ (x, y) \in \C^2 \mid x=-y\}$). By Lemma \ref{check}, it suffices to show that $L_1$ and $L_2$ are bialgebraic for $\wp_{\Lambda \times \bar{\Lambda}}$.  Write $\Lambda=\langle \tau_1, \tau_2 \rangle_\mathbb Z$, and note that, for $i=1, 2$, we have $\gamma {\tau}_i \in \Lambda$ because $\gamma \in \mathbb Z$. Therefore, the elements $w_i=(\gamma \tau_i, \gamma \tau_i)$, for $i=1, 2$, belong to $\Lambda \times \overline{\Lambda}$; furthermore, we have $L_1=\langle w_1,  w_2 \rangle_\mathbb R$. Similarly, letting $w_i'=(-\gamma \tau_i, \gamma \tau_i) \in \Lambda \times \overline{\Lambda} $ for $i=1, 2$, we have $L_2=\langle w_1',  w_2' \rangle_\mathbb R$. Therefore, Theorem \ref{thm:pila} implies that $L_1$ and $L_2$ are bialgebraic for $\wp_{\Lambda \times \bar{\Lambda}}$.

If $\gamma \not \in \Q$, letting $\Lambda'=\sqrt{\gamma}\Lambda$ we have $a \Lambda'=\sqrt{\bar{\gamma}}\gamma\Lambda\subset\sqrt{\bar{\gamma}}\bar{\Lambda}=\overline{\Lambda'}$. Therefore, Theorem \ref{mainthmwp}(2)(b) holds true for $\Lambda'$. Furthermore, multiplication by $\sqrt{\gamma}$ on $\R^2 \simeq \C$ sends bialgebraic subsets for $\mathcal{P}_\Lambda$ to bialgebraic subsets for $\mathcal{P}_{\Lambda'}$. This yields a bijection between the sets of bialgebraic sets for $\mathcal{P}_{\Lambda}$ and $\mathcal{P}_{\Lambda'}$; furthermore, multiplication by $\sqrt{\gamma}$ sends the lines $\mathcal{L}_1, \mathcal{L}_2$ in the statement of Theorem \ref{mainthmwp}(2)(b) to the $x$ and $y$ axes respectively.
Hence, the statement for $\Lambda$ follows from the statement for $\Lambda'$.

\subsection{Proof of Theorem \ref{mainthmwp} for CM lattices}\label{subsec:proofCM}
Let us prove Theorem \ref{mainthmwp}(3). Let $\Lambda$ be a CM lattice and $K=\Isog(\Lambda,\Lambda) \otimes \mathbb Q\subset \C$; to prove our statement, we can replace $\Lambda$ by $\alpha \Lambda$ for any $\alpha \in \C^\times$. Without loss of generality, let us work with $\Lambda=\langle 1, \tau \rangle_\Z$ for some $\tau$ in the upper half-plane; we have $K=\Q(\tau)=\Isog(\Lambda, \bar{\Lambda})\otimes \Q$.

Let $\cV\subset \R^2$ be a bialgebraic set that is not a singleton. We know from \S \ref{sectionr} that $\cV$ is a translate of a line $\mathcal{L} \subset \R^2$ through the origin with the following property: there exists $r \in \C$ of absolute value one such that every $(x, y) \in \mathcal{L}$ satisfies $x-iy=r(x+iy)$. Furthermore, by \eqref{double} we have $r \in K$.

Any line $\mathcal{L}$ as above intersects $\Lambda$ at a non-zero point. Indeed, by Hilbert's Theorem 90 there exists $\rho \in K^\times$ such that $r=\frac{\rho}{\overline{\rho}}$, and $\mathcal{L}$ is the line $\rho^{-1} \R \subset \C \simeq \R^2$. There exists a non-zero integer $N$ such that $N \rho^{-1} \in \Lambda$, therefore $N \rho^{-1}$ belongs to $\Lambda \cap \mathcal{L}$.

Conversely, if $\mathcal{L} \subset \R^2$ is a line through the origin containing a non-zero point of $\Lambda$, then there exists $\rho \in K^\times$ such that $\mathcal{L}=\rho^{-1} \R$. Let $r=\frac{\rho}{\overline{\rho}}$ and $L=\{ (x, y) \in \C^2: y=rx \}$. Reversing the argument in the previous paragraph, we see that $f(\mathcal{L}) \subset L$. As $\Isog(\Lambda, \bar{\Lambda})\otimes \Q=K$, there is a non-zero integer $N$ such that $Nr \in \Isog(\Lambda, \overline{\Lambda})$. Letting $w_1=(N, Nr) \in \Lambda \times \overline{\Lambda}$ and $w_2=(N\tau, Nr\tau) \in \Lambda \times \overline{\Lambda}$, we have $L=\langle w_1,  w_2 \rangle_\R$. Theorem \ref{thm:pila} implies that $L$ is bialgebraic for $\wp_{\Lambda \times \overline{\Lambda}}$, therefore the line $\mathcal{L}$ is bialgebraic for $\mathcal{P}$ by Lemma \ref{check}, and the same holds true for its translates by Lemma \ref{translates}.

\begin{remark}
Let $\Lambda\subset \C$ be a lattice. The fact that bialgebraic lines for $\mathcal{P}_\Lambda$ must be translates of lines through the origin meeting $\Lambda$ at a non-zero point can be proved directly, via purely topological considerations. Indeed, if $\mathcal{L}$ is a line such that $\Lambda\cap \mathcal{L}=\{0\}$ then the image of $\mathcal{L}$ in $\C/\Lambda$ is dense, hence the same is true for the image of $\mathcal{L} \cap (\R^2 \smallsetminus \Lambda)$ via $\mathcal{P}$ in $\R^2\simeq \C \subset \Pr^1(\C)$. Therefore, the line $\mathcal{L}$ and its translates cannot be bialgebraic. Theorem \ref{mainthmwp} tells us that in the CM case - and in no other case - this topological constraint is the only obstruction to bialgebraicity of lines.
\end{remark}

\subsection{A group theoretic interpretation of bialgebraic lines and their images}\label{ssec:gpbialgell} Let $\Lambda$ be a lattice in $\C$. The map $\C \rightarrow \Pr^2(\C)$ sending $z \in \C \smallsetminus \Lambda$ to $[\wp_{\Lambda}(z): \wp'_{\Lambda}(z): 1]$ and $\Lambda$ to $O=[0: 1: 0]$ has image the complex points of a curve $E_\Lambda\subset \Pr^2_\C$; let
\begin{equation*}
p \colon \C\simeq \R^2 \rightarrow E_\Lambda(\C)=\mathrm{Res}_{\C/\R}E_\Lambda(\R)
\end{equation*}
be the resulting map, and let $Y=\Res_{\C/\R}E_\Lambda$. In this section, we will reinterpret the previous results in terms of the uniformization map $p$, and show that bialgebraic curves arise from elliptic curves defined over the reals and isogenous to $E_\Lambda$.
\begin{definition}\label{def:balgell} \leavevmode
    \begin{enumerate}
        \item A non-empty subset $\mathcal{V}\subsetneq \R^2$ is bialgebraic for $p$ if it satisfies property (2) of Definition \ref{def-bialg}, as well as the following variant of property (1): there exist algebraic subvarieties $V \subset \A^2_\R$ and $W \subsetneq Y$ such that $\cV=V(\mathbb R)$ and $p(\mathcal{V})\subset W(\R)$.
        \item An arc in a plane curve $\hat{C} \subset \A^2_\R$ is an open neighbourhood $\mathcal{I}\subset \hat{C}(\R)$ (in the Euclidean topology) of a smooth real point of $\hat{C}$, homeomorphic to an open interval. A subset $\mathcal{C}\subset E_\Lambda(\C)=Y(\R)$ is called a bialgebraic curve if there is an irreducible curve $C \subset Y$ such that $\mathcal{C}=C(\R)$, an algebraic curve $\hat{C} \subset \A^2_\R$ and an arc $\mathcal{I} \subset \hat{C}(\R)$ such that $p(\mathcal{I})\subset \mathcal{C}$.
    \end{enumerate}
\end{definition}

\begin{remark}\label{rem:defbalg}
Let us give some motivation behind Definition \ref{def:balgell}(2) - which will be used starting from Remark \ref{rem-bialgrealell} - and spell out the relation between bialgebraic curves in $Y(\R)$ and in $\R^2$.
\begin{enumerate}
    \item If $\mathcal{V}=V(\R)$ is a bialgebraic set that is not a singleton and $p(\cV)\subset W(\R)$, a priori $p(\cV)$ might not be contained in an irreducible component of $W$. On the other hand, there is an arc $\mathcal{I}\subset V(\R)$ such that $p(\mathcal{I})$ is contained in an irreducible component of $W$.
    \item If $\mathcal{C}=C(\R)$ is a bialgebraic curve, then we may take the curve $\hat{C}$ in Definition \ref{def:balgell}(2) to be irreducible. The arguments in the proof of Lemma \ref{lem:realcompbalg} and in \S \ref{ssec:linbalg} show that $\mathcal{I}$ is an arc in a translate $\mathcal{V}$ of a line as in Theorem \ref{mainthmwp}. By analytic continuation we deduce that $p(\mathcal{V})\subset \mathcal{C}$.
\end{enumerate}
\end{remark}

\begin{lemma}\label{lem-bialgpvsP}
    A set $\mathcal{V} \subset \R^2$ is bialgebraic for $\mathcal{P}_\Lambda$ if and only if it is bialgebraic for $p$.
\end{lemma}

\begin{proof}
The complement $E_\Lambda^\circ=E_\Lambda \smallsetminus \{O\}$ is an affine cubic plane curve with equation $y^2=x^3+ax+b$. Let $p_x^\circ: E^\circ_\Lambda \rightarrow \A^1_\C$ be the projection on the first coordinate; note that the composite of the map $\R^2 \smallsetminus \Lambda \rightarrow \mathrm{Res}_{\C/\R}E^\circ_\Lambda(\R)$ induced by $p$ and of the map $q=\mathrm{Res}_{\C/\R}(p_x^\circ)$ is the map $\mathcal{P}_\Lambda$, and that the morphism $q$ is finite and surjective.

Let $\mathcal{V}\subset \R^2$ be a subset satisfying condition (2) of Definition \ref{def-bialg}. Note that $\mathcal{V}$ is bialgebraic for $\mathcal{P}_\Lambda$ (resp. for $p$) if and only if $\mathcal{P}_\Lambda(\mathcal{V} \cap (\R^2\smallsetminus \Lambda))$ is not Zariski dense in $\A^2_\R$ (resp. $p(\mathcal{V})$ is not Zariski dense in $Y$). As $q$ is surjective, if $p(\mathcal{V})$ is Zariski dense in $Y$ then $\mathcal{P}_\Lambda(\mathcal{V} \cap (\R^2 \smallsetminus \Lambda))$ is Zariski dense in $\A^2_\R$. On the other hand, if $p(\mathcal{V})$ is not Zariski dense in $Y$ then it is contained in $W(\R)$ for a closed subset $W \subset Y$ of dimension at most one. As $q$ is finite, the set $q(W\cap \mathrm{Res}_{\C/\R}E_\Lambda^\circ)\subset \A^2_\R$ is closed and of dimension at most one, and it contains $\mathcal{P}_\Lambda(\mathcal{V} \cap (\R^2 \smallsetminus \Lambda))$.
\end{proof}
\subsubsection{The case of self-conjugate lattices}\label{subsec-bialgselfcong} If $\Lambda=\overline{\Lambda}$, then the curve $E_\Lambda^\circ=E_\Lambda \smallsetminus \{O\}$ has equation $y^2=x^3+ax+b$ with $a, b \in \R$. In this situation, the fact that the coordinate axes in $\R^2$ are bialgebraic for $p$ can be explained in group-theoretic terms. Indeed, let $E_{\Lambda, \R}$ be the real elliptic curve such that $E_{\Lambda, \R}^\circ$ has Weierstrass equation $y^2=x^3+ax+b$. The image in $E_\Lambda(\C)$ of the horizontal axis is an oval in $E_{\Lambda, \R}(\R)$ containing the point at infinity (for completeness we observe that, if $\Lambda=\langle \tau_1, \tau_2 \rangle_\Z$ is a rectangular lattice with $\tau_1 \in \R$ and $\tau_2 \in i\R$, then the line $\R \tau_1+\frac{1}{2}\tau_2$ projects to another oval in $E_{\Lambda, \R}(\R)$, cf. \cite[pp. 285, 286]{was08}).

On the other hand, the image of the vertical axis via $p$ is contained in the set of fixed points of the anti-holomorphic involution $P \mapsto -\bar{P}$. This is the set of real points of the quadratic twist $\tilde{E}_{\Lambda, \R}$, such that $\tilde{E}_{\Lambda, \R}^\circ$ has equation $-y^2=x^3+ax+b$, mapping to $E_\Lambda(\C)$ via $(x, y)\mapsto (x, iy)$. In other words, the Zariski closures of the images of the coordinate axes via $p$ are the sets of real points of the subgroups $E_{\Lambda, \R}\subset \mathrm{Res}_{\C/\R}(E_{\Lambda})$ and $\tilde{E}_{\Lambda, \R}\subset \mathrm{Res}_{\C/\R}(E_{\Lambda})$.

We will now argue that bialgebraicity of all the lines in Theorem \ref{mainthmwp} is deduced from the example we just described, taking images via automorphisms of the curve $E_\Lambda$ or via an isogeny from another curve.

\begin{lemma}\label{lem:bialg-isog}
    Let $\Lambda$ and $\Lambda'$ be two lattices in $\C$, and let $p: \C \rightarrow E_\Lambda(\C)$ and $p': \C \rightarrow E_{\Lambda'}(\C)$ be the uniformization maps. If $\Lambda' \subset \Lambda$ then a set $\cV \subset \R^2$ is bialgebraic for $p'$ if and only if it is bialgebraic for $p$.
\end{lemma}
\begin{proof}
    This follows, as in the proof of Lemma \ref{lem-bialgpvsP}, from the fact that the quotient map $\C/\Lambda' \rightarrow \C/\Lambda$ is attached to a finite surjective morphism $E_{\Lambda'}\rightarrow E_\Lambda$, hence the induced morphism $\mathrm{Res}_{\C/\R}E_{\Lambda'}\rightarrow \mathrm{Res}_{\C/\R}E_\Lambda$ is finite surjective.
\end{proof}

\subsubsection{The case of Theorem \ref{mainthmwp}(2)(b)}\label{subsubsec-bialgcvncm} Let $\Lambda$ be a lattice as in Theorem \ref{mainthmwp}(2)(b). As explained in \S \ref{subsec: proof(2)}, up to replacing $E_\Lambda$ by an isomorphic elliptic curve we may assume that $\Isog(\Lambda, \bar{\Lambda})=\Z\cdot \gamma$ with $\gamma \in \Q$. In this case, the intersection $\Lambda'=\Lambda \cap \overline{\Lambda}$ is a sublattice of $\Lambda$, and is self-conjugate. Bialgebraicity of the coordinate axes for $p'$ follows from \S \ref{subsec-bialgselfcong}, and Lemma \ref{lem:bialg-isog} implies that the coordinate axes are bialgebraic for $p$.

\subsubsection{The CM case}\label{subsubsec-bialgcvcm}

Let $\Lambda=\langle 1, \tau \rangle_\Z$ be a CM lattice such that $\Lambda=\bar{\Lambda}$ (every CM lattice is isogenous to one with this property). Let $\mathcal{L}\subset \R^2$ be the $x$-axis. As explained in \S \ref{subsec-bialgselfcong}, the set $\mathcal{L}$ is bialgebraic for $p$. Furthermore, the argument in \S \ref{subsec:proofCM} shows that every bialgebraic line through the origin is of the form $\rho^{-1} \mathcal{L}$ for some element $\rho \in K^\times$, where $K=\mathrm{Isog}(\Lambda, \Lambda)\otimes \Q$. Up to replacing $\rho$ by a non-zero rational multiple, we may assume that $ \rho^{-1}\Lambda \subset \Lambda$, hence bialgebraicity of $\rho^{-1}\mathcal{L}$ for $p$ follows from bialgebraicity of $\mathcal{L}$ and from Lemma \ref{lem:bialg-isog}.

\begin{remark}\label{rem-bialgrealell}
    Let $\Lambda \subset \C$ be a lattice and let $\mathcal{C}=C(\R)\subset Y(\R)$ be a bialgebraic curve. Using Remark \ref{rem:defbalg} we see that the curve $C$ is a translate of the Zariski closure of $p(\mathcal{V})$ for some bialgebraic line $\mathcal{V}\subset \R^2$  passing through the origin. It follows from the above discussion that there is a real elliptic curve $E'$ and an isogeny $\phi: E'\times_{\Spec(\R)} \Spec(\C) \rightarrow E_\Lambda$ such that the Zariski closure of $p(\cV)$ is the image of $E'$ via $\mathrm{Res}_{\C/\R}\phi$. As a consequence, if $C$ contains the identity point of $E_\Lambda$ then it is a real elliptic curve in $\mathrm{Res}_{\C/\R}E_\Lambda$.
\end{remark}

We collect some consequences of the above discussion in the following theorem, which will be extended to hyperbolic curves in 
$\S$\ref{sect:hypcurves}. We call a bialgebraic set for $p$ that is not a singleton a \emph{bialgebraic curve for} $p$. We denote by $\mathrm{Bialg}_Y$ the set of bialgebraic curves in $Y(\R)=\mathrm{Res}_{\C/\R}E_\Lambda(\R)$. Note that this set is stable by translation by any element of $E_\Lambda(\C)$.

\begin{theorem}\label{mainthmell}
Let $\Lambda\subset \C$ be a lattice with attached elliptic curve $E_\Lambda$, and let $p: \C \rightarrow E_\Lambda(\C)$ be the uniformization map. There exist bialgebraic curves for $p$ if and only if $E_\Lambda$ is isogenous to an elliptic curve with real $j$-invariant. Furthermore, the following properties are equivalent:
\begin{enumerate}
    \item the elliptic curve $E_\Lambda$ has complex multiplication;
    \item for every closed geodesic $\mathcal{G}\subset E_\Lambda(\C)$, the set of real points of the Zariski closure of $\mathcal{G}$ in $\mathrm{Res}_{\C/\R}E_\Lambda$ is a bialgebraic curve;
    \item the quotient of the set $\mathrm{Bialg}_Y$ by the equivalence relation given by translation is infinite.
\end{enumerate}
\end{theorem}
\begin{proof}
The discussion in \S \ref{subsubsec-bialgcvncm} and \S \ref{subsubsec-bialgcvcm} shows that, if there are bialgebraic curves for $p$, then $E_\Lambda$ is isogenous to an elliptic curve attached to a self-conjugate lattice, hence having real $j$-invariant. To prove the converse, by Lemma \ref{lem:bialg-isog} it suffices to show that, if $E_\Lambda$ has real $j$-invariant, then there are bialgebraic curves for $p$. If $j(E_\Lambda)$ is real then $E_\Lambda$ is isomorphic to a curve attached to a lattice $\Lambda_\tau=\Z+\tau\Z$ with $\tau$ of one of the following forms (cf. \cite[Chapter V, Proposition 2.1]{sil94}): $\tau=it$ for some real number $t \geq 1$, or $\tau=e^{i \theta}$ for some real number $\frac{\pi}{3}\leq \theta \leq \frac{\pi}{2}$, or $\tau=\frac{1}{2}+it$ for some real number $t \geq \frac{\sqrt{3}}{2}$. In the first and third case the lattice $\Lambda_\tau$ is self-conjugate, therefore there are bialgebraic curves for $p$. Such curves exist in the second case as well, as $E_{\Lambda_{\tau}}$ is isomorphic to the curve attached to a self-conjugate lattice $\Lambda_{\tau'}$, for $\tau'$ of the form $\frac{1}{2}+it$ with $\frac{1}{2}\leq t \leq \frac{\sqrt{3}}{2}$ \cite[p. 417]{sil94}.

Let us now prove the equivalence of properties (1), (2) and (3).
\begin{description}
    \item[$(1)\Rightarrow (2)$] a closed geodesic $\mathcal{G}$ is the image of a line $\mathcal{L}\subset \R^2$ as in Theorem \ref{mainthmwp}(3). Hence, we have $\mathcal{G}\subset C(\R)$ for a curve $C \subset \mathrm{Res}_{\C/\R}E_\Lambda$. As in Remark \ref{rem:defbalg} we see that $\mathcal{G}$ is contained in an irreducible component $C'$ of $C$, and the set $C'(\R)$ is a bialgebraic curve.
    \item[$(2) \Rightarrow (1)$] if $E_\Lambda$ does not have complex multiplication, take a closed geodesic $\mathcal{G} \subset E_\Lambda(\C)$ that is the image of a line through the origin different from the lines in Theorem \ref{mainthmwp}(2). The theorem and Lemma \ref{lem-bialgpvsP} imply that $\mathcal{G}$ is Zariski dense in $\mathrm{Res}_{\C/\R}E_{\Lambda}$.
    \item[$(1) \Rightarrow (3)$] there is an elliptic curve $E_{\Lambda'}$ attached to a self-conjugate lattice and an isogeny $q: E_{\Lambda'}\rightarrow E_{\Lambda}$. Let $C\subset \mathrm{Res}_{\C/\R}E_\Lambda$ be the image of $E_{\Lambda', \R}$ (defined in \S \ref{subsec-bialgselfcong}) via $\mathrm{Res}_{\C/\R}q$. It is a subgroup scheme of $\mathrm{Res}_{\C/\R}E_\Lambda$, and $\mathcal{C}=C(\R)$ is a bialgebraic curve. For every isogeny $\phi : E_{\Lambda}\rightarrow E_{\Lambda}$ the real points of $\Res_{\C/\R}\phi(C)$ form a bialgebraic curve, which is not a translate of $\mathcal{C}$ unless it is equal to it. The union of the curves $\Res_{\C/\R}\phi(C)$, as $\phi$ varies among all self-isogenies of $E_{\Lambda}$, is dense in $\mathrm{Res}_{\C/\R}E_\Lambda$; in particular, there are infinitely many such curves.
    \item[$(3)\Rightarrow (1)$] in view of Remark \ref{rem:defbalg}, this follows from Theorem \ref{mainthmwp}. \qedhere
\end{description}
\end{proof}

\subsection{A modular criterion for the existence of bialgebraic curves}\label{subs:modcrit}

\subsubsection{} Let $\Lambda\subset \C$ be a lattice. The existence of bialgebraic curves for $p: \C \rightarrow E_\Lambda(\C)$ (equivalently, for $\mathcal{P}_\Lambda$) requires two main inputs.
\begin{enumerate}
    \item One needs $E_\Lambda$ to be isogenous to $E_{\bar{\Lambda}}$ to guarantee that there are complex bialgebraic curves for $\wp_{\Lambda \times \overline{\Lambda}}$ that are not horizontal or vertical lines; cf. \S \ref{subsec:lamdanotisogbar}.
    \item The existence of bialgebraic curves for $\wp_{\Lambda \times \overline{\Lambda}}$ as above pulling back to curves in $\R^2$ requires the more restrictive condition that there is a non-zero element in $\Isog(\Lambda, \bar{\Lambda})$ with rational absolute value (this was established using \S \ref{sectionr} and Lemma \ref{rvalue} in the non-CM case; in the CM case, such an element always exists), and is in turn equivalent to the fact that $E_{\Lambda}$ is isogenous to an elliptic curve defined over the reals by Theorem \ref{mainthmell}.
\end{enumerate}

Note that $\Isog(\Lambda, \overline{\Lambda})$ is non-zero, respectively contains a non-zero element with rational absolute value, if and only if the same property is true for $\Isog(\alpha \Lambda, \overline{\alpha \Lambda})$, for any $\alpha \in \C^\times$. Let us therefore restrict to lattices of the form $\langle 1, \tau \rangle_\Z$, with $\tau$ belonging to the upper half-plane $\Hp$. In this section, we will describe the subsets of $\mathbb{H}$ corresponding to lattices satisfying conditions (1) and (2) above, in terms of the special geodesics studied in \cite[\S 3.3]{Tam23}.

\begin{proposition}\label{prop:locus}
    Let $\Lambda=\langle 1, \tau \rangle_{\Z}$ be a lattice with $\tau \in \mathbb H$.
    \begin{enumerate}
        \item The elliptic curve $E_{\Lambda}$ is isogenous to $E_{\overline{\Lambda}}$ if and only if $\tau$ belongs to a geodesic with endpoints in $\Pr^1(\Q)$, or with conjugate real quadratic endpoints.
        \item The elliptic curve $E_{\Lambda}$ is isogenous to an elliptic curve defined over the reals if and only if $\tau$ belongs to a geodesic with endpoints in $\mathbb P^1(\mathbb Q)$.
    \end{enumerate}
\end{proposition}

\begin{proof}\leavevmode
\begin{enumerate}
    \item If there is a non-zero element in $\Isog(\Lambda, \overline{\Lambda})$ then $\tau=\frac{a \bar{\tau}+b}{c \bar{\tau}+d}$ for some matrix $\mathsf{A}=\begin{pmatrix}
    a & b\\
    c & d
\end{pmatrix} \in M_2(\Z)$ of trace zero and negative determinant. Hence, $\tau$ belongs to a geodesic as in point (1) of the proposition - cf. \cite[Remark 3.3.6, Proposition 3.3.7]{Tam23}. Conversely, by the same remark and proposition, if $\tau$ belongs to such a geodesic then $\mathsf{A} \cdot \bar{\tau}=\tau$ for a matrix $\mathsf{A}$ as above (the action of $\mathsf{A}$ being via Möbius transformations), hence $c\bar{\tau}+d$ belongs to $\Isog(\Lambda, \overline{\Lambda})$.
    \item Let $\mathcal{R}\subset \Hp$ be the set of numbers of the form $it$ for some real number $t \geq 1$, or $e^{i \theta}$ for some real number $\frac{\pi}{3}\leq \theta \leq \frac{\pi}{2}$, or $\frac{1}{2}+it$ for some real number $t \geq \frac{\sqrt{3}}{2}$. By the description of real elliptic curves recalled in the proof of Theorem \ref{mainthmell}, the curve $E_{\Lambda}$ is isogenous to an elliptic curve defined over the reals if and only if there is $\tau'$ in $\mathcal{R}$ and a matrix $\mathsf{A} \in M_2(\Z)$ with positive determinant such that $\tau=\mathsf{A} \cdot \tau'$. As $\mathcal{R}$ is contained in a union of geodesics with rational endpoints, if $E_\Lambda$ is isogenous to an elliptic curve defined over $\R$ then $\tau$ belongs to a geodesic with rational endpoints. Conversely, if $\tau$ belongs to such a geodesic then it is of the form $\mathsf{A} \cdot \tau'$ for some $\tau'$ on the positive imaginary axis and some $\mathsf{A} \in M_2(\Z)$ with positive determinant, therefore $E_{\Lambda}$ is isogenous to an elliptic curve defined over the reals. \qedhere
\end{enumerate}
\end{proof}

\section{Subvarieties of real abelian varieties and torsion points}\label{sec-rrayn}

In this section, as the title suggests, we observe that the main result in \cite{ray83} has an analog for real abelian varieties. In the particular case of a Weil restriction $Y=\mathrm{Res}_{\C/\R}E$ of a complex elliptic curve, we deduce a relation between curves in $Y$ containing infinitely many torsion points and bialgebraic curves, mirroring what happens in the complex setting (cf. \cite[Example 4.6]{kuy18}).

\subsection{A real version of Raynaud's theorem}
\begin{theorem}[Raynaud\cite{ray83}]\label{thm-rayn}  Let $A$ be a complex abelian variety and let $Z \subset A$ be an irreducible subvariety. If $Z \cap A(\C)_{\tors}$ is Zariski dense in $Z$, then $Z$ is the translate of an abelian subvariety of $A$ by a point in $A(\C)_{\tors}$. 
\end{theorem}

Theorem \ref{thm-rayn} implies the following real analog.

\begin{proposition}
    Let $A$ be a real abelian variety, and let $W\subset A$ be an irreducible subvariety. If $A(\R)_{\tors}\cap W$ is Zariski dense in $W$, then $W$ is the translate of a real abelian subvariety of $A$ by a point in $A(\R)_{\tors}$.
\end{proposition}
\begin{proof}
Up to translating $W$ by a point of $A(\R)_{\tors}\cap W$, we may assume that it contains the identity. The base change $W_\C$ of $W$ to the complex numbers is an irreducible complex subvariety of $A_\C$ \cite[\href{https://stacks.math.columbia.edu/tag/0G69}{Tag 0G69}]{stacks-project} containing a Zariski dense subset of torsion points. By Theorem \ref{thm-rayn}, we can write $W_\C=A'+\tau$ for some torsion point $\tau$ and some abelian subvariety $A'$ of $A_\C$. As $W$ contains the identity element, we have $-\tau \in A'$, therefore $W_\C=A'$. The intersection $A(\R)\cap A'$ is Zariski dense in $A'$ by hypothesis, and it is fixed by the action of complex conjugation on $A_\C$. Therefore, this action sends $A'$ to itself, and yields a descent datum for $A'$ to the real numbers. Hence, there exists a subvariety $B$ of $A$ such that $A'=B_\C$ - so that $A(\R)\cap A'=B(\R)$ - and we have $W=B$ \cite[Proposition 4.1, Proposition 4.3]{mil24}. To conclude, let us show that $B$ is an abelian subvariety of $A$. Let $m: A \times A \rightarrow A$ be the multiplication map; the base change $m_\C: A_\C \times A_\C \rightarrow A_\C$ sends $A(\R)\times A(\R)$ to $A(\R)$, and $A'\times A'$ to $A'$. Therefore, we have $m(B(\R) \times B(\R))\subset B(\R)$. As $B(\R)\times B(\R)$ is Zariski dense in $B \times B$, we deduce that $m$ sends $B \times B$ to $B$. For the same reason the inversion map $\iota: A \rightarrow A$ restricts to a map $B \rightarrow B$. Hence, $B$ is an abelian subvariety of $A$.
\end{proof}

\subsection{Example: Weil restriction of complex elliptic curves}

Let $E$ be an elliptic curve over $\C$ and let $Y=\Res_{\C/\R} E$; the group structure on $E$ induces one on $Y$, which is therefore a real abelian surface.
\begin{theorem}\label{thm:mmrealell}
    Let $W\subset Y$ be an irreducible curve. The following assertions are equivalent.
    \begin{enumerate}
        \item The intersection $W\cap Y(\R)_{\tors}$ is infinite.
        \item There exists a real elliptic curve $E' \subset Y$ and a point $\alpha \in Y(\R)_{tors}$ such that $W=E'+\alpha$.
        \item There exists a bialgebraic curve $\mathcal{C}= C(\R)\subset Y(\R)$ passing through the origin and a point $\alpha \in Y(\R)_{\tors}$ such that $W=C+\alpha$. 
    \end{enumerate}
\end{theorem}

\begin{proof}
    The implication $(1) \Rightarrow (2)$ is a special case of the previous proposition. To prove that $(2) \Rightarrow (3)$, it suffices to show that $E'(\R)$ is a bialgebraic curve. Let $p: \R^2 \rightarrow Y(\R)$ be the uniformization map. The inclusion $E' \subset Y$ induces on tangent spaces at the origin a linear map $T_{O}(E')\rightarrow T_{O}(Y)=\R^2$, whose image $\mathcal{L}$ is a line such that $p(\mathcal{L})\subset E'(\R)$. Therefore, the set $E'(\R)$ is a bialgebraic curve. Finally, let us prove that $(3) \Rightarrow (1)$. Take a curve $W$ as in $(3)$; we claim that $W \cap Y(\R)_{\tors}$ is infinite. Indeed, we may assume that $W=C$ for a curve $C \subset Y$ through the origin such that $C(\R)$ is bialgebraic, and the claim follows from Remark \ref{rem-bialgrealell}.
\end{proof}

\begin{remark}
    For completeness, let us mention that the set of real elliptic curves $E' \subset Y$ can be described as follows. Given a line $\mathcal{L}\subset \R^2$ through the origin that is bialgebraic for $p$, the Zariski closure of $p(\mathcal{L})$ is an elliptic curve in $Y$ by Remark \ref{rem-bialgrealell}. The resulting function $\psi$ from the set of bialgebraic lines through the origin in $\R^2$ to the set of elliptic curves in $Y$ is surjective, by the argument in the proof of the implication $(2)\Rightarrow (3)$ in Theorem \ref{thm:mmrealell}. Furthermore, the function $\psi$ is injective. Indeed, if $\mathcal{L}_1$ and $\mathcal{L}_2$ are two distinct bialgebraic lines through the origin then $p(\mathcal{L}_1)$ and $p(\mathcal{L}_2)$ are distinct ovals with non-empty intersection in $Y(\R)$. Therefore, they cannot be contained in the set of real points of one elliptic curve $E' \subset Y$.

    As a consequence we see that in the situation of Theorem \ref{mainthmwp}(1) or (2)(a) there are no elliptic curves in $Y$, whereas in the situation of Theorem \ref{mainthmwp}(2)(b) the abelian surface $Y$ contains two elliptic curves (the two real forms of $E$). Finally, if $E$ has complex multiplication then the set of elliptic curves in $Y$ is infinite. More precisely, if $E$ has CM by an imaginary quadratic field $K$, the above discussion and the argument in \S \ref{subsec:proofCM} yield a bijection between the set of elliptic curves in $Y$ and the set of elements of $K^\times$ with absolute value one.
\end{remark}

\section{Real bialgebraic curves in hyperbolic algebraic Riemann surfaces}\label{sect:hypcurves}

The aim of this section is to describe real bialgebraic curves for the uniformization map of a hyperbolic algebraic (connected) Riemann surface.

\subsection{Setting and main statement}\label{sec-sett} Let $X$ be a hyperbolic algebraic curve, i.e., a complex algebraic curve obtained by removing $m\geq 0$ points from a smooth, projective, connected curve of genus $g$ with $2g-2+m>0$. The set of complex points $X(\C)$ is naturally equipped with the structure of a Riemann surface and, by the uniformization theorem for Riemann surfaces, the universal cover of $X(\C)$ is the upper half-plane $\Hp$. Let $\PGL(\R)^+\subset \PGL(\R)$ be the subgroup of matrices with positive determinant, acting on $\Hp$ via Möbius transformations. The Riemann surface $X(\C)$ is biholomorphic to $\Gamma \backslash \Hp$, for some lattice $\Gamma \subset \PGL(\R)^+$. We denote by $p$ the composite of the quotient map $\Hp \rightarrow \Gamma \backslash \Hp$ and of a fixed biholomorphism $\Gamma \backslash \Hp\xrightarrow{\sim} X(\C)$.

Let $Y=\Res_{\C/\R}X$. Via the natural identification $Y(\R)\simeq X(\C)$ we obtain a map $\Hp \rightarrow Y(\R)$, still denoted by $p$. The following definition of bialgebraic curves in $Y(\R)$ mirrors Definition \ref{def:balgell}.

\begin{definition}\label{def-bialgh} \leavevmode
\begin{enumerate}
\item A non-empty subset $\mathcal{V} \subset \Hp$ is called algebraic if there is a subvariety $V \subset \A^2_\R$ such that $\mathcal{V}=V(\R) \cap \Hp$, and $\mathcal{V}$ cannot be written as a union of two proper subsets of the form $V'(\R)\cap \Hp$, where $V' \subset \A^2_\R$ is a subvariety. An algebraic subset $\mathcal{V} \subsetneq \Hp$ that is not a singleton is called an algebraic curve, and an open neighbourhood $\mathcal{I}\subset \mathcal{V}$ (in the Euclidean topology) of a smooth real point of $V$, homeomorphic to an open interval, is called an arc in $\mathcal{V}$.
\item A subset $\mathcal{C} \subset Y(\R)$ is called a bialgebraic curve if there is an irreducible algebraic curve $C \subset Y$ such that $\mathcal{C}=C(\R)$, and there exists an arc $\mathcal{I}$ in an algebraic curve $\hat{\mathcal{C}}\subset \Hp$ such that $p(\mathcal{I})\subset \mathcal{C}$.
\end{enumerate}
\end{definition}

\begin{remark}\label{rem-alginH}\leavevmode
\begin{enumerate}
\item In the setting of (2), we may assume that $\hat{\mathcal{C}}=\hat{C}(\R) \cap \Hp$ for an irreducible curve $\hat{C}\subset \A^2_\R$, vanishing locus of an irreducible real polynomial in two variables.
\item As $X$ is hyperbolic, every biholomorphism $a : X(\C)\rightarrow X(\C)$ is algebraic \cite[Theorem 1.2(ii)]{jk20}. Furthermore, the map $a$ lifts to a biholomorphism of $\Hp$, which preserves algebraic subsets of $\Hp$ (as it extends to an automorphism of $\Pr^1_\C$). Hence, if $\mathcal{C}\subset Y(\R)$ is a bialgebraic curve then so is $a(\mathcal{C})$; furthermore, bialgebraic curves in $Y(\R)$ do not depend on the chosen biholomorphism $\Hp/\Gamma\xrightarrow{\sim} X(\C)$.
\end{enumerate}
\end{remark}

\subsubsection{Example: real algebraic curves}\label{ex-realbalg} Suppose that $X$ is the base change of a real algebraic curve $C$. The morphism $C \rightarrow Y=\mathrm{Res}_{\C/\R}X$ adjoint to the identity induces the natural inclusion $C(\R)\rightarrow C(\C)=Y(\R)$. If $\mathcal{C}=C(\R)$ is not empty, then it is a bialgebraic curve in $Y(\R)$. Indeed, letting $\D\subset \Pr^1(\C)$ be the complement of $\Pr^1(\R)$, the universal cover $p: \Hp \rightarrow X(\C)$ extends to a holomorphic cover $p': \D \rightarrow X(\C)$, equivariant for the action of complex conjugation on source and target. The automorphism group $\Gamma'$ of the equivariant covering $p'$ embeds in $\PGL(\R)$ via its action on $\D$, and it is a split extension of $\Gamma$ by the cyclic group of order 2, cf. \cite[Proposition 4.2 and its proof]{hui01}. The equivariance of $p'$ implies that, if $z \in \Hp$ is such that $\bar{z}=\gamma z$ for some $\gamma \in \Gamma'$, then $p'(z)$ belongs to $C(\R)$. Taking $\gamma \in \Gamma'$ of order 2 we obtain an algebraic curve $\hat{\mathcal{C}}=\{z \in \Hp \mid \bar{z}=\gamma z\}$ in $\Hp$ such that $p(\hat{\mathcal{C}})\subset \mathcal{C}$.

\vspace{10pt}

Our first aim is to prove the following result, which may be thought of as analogous to the equivalence between points (1) and (3) in Theorem \ref{mainthmell}.

\begin{theorem}\label{thm-main}
The lattice $\Gamma \subset \PGL(\R)^+$ is arithmetic if and only if there are infinitely many bialgebraic curves in $Y(\R)$.
\end{theorem}

\begin{remark}\leavevmode \label{rem:finnonar}
\begin{enumerate}
    \item The notion of arithmetic lattice in the Lie group $\PGL(\R)^+$ is defined in the statement of Margulis' theorem \cite[Theorem 1']{mar91}. This result is the key input in the proof of one of the implications in Theorem \ref{thm-main}. We note that Margulis' theorem has also been used to prove a different finiteness result for hyperbolic curves in \cite{moc98}.
    \item A finiteness result for maximal totally geodesic subvarieties of certain locally symmetric spaces with non-arithmetic fundamental group was proved in \cite[Theorem 1.2.1]{bau23} (as well as in \cite{bfms23}). This was one of the reasons, together with the finiteness statement \cite[Problem 5]{cw02} of more arithmetic flavor, why we were led to guess that Theorem \ref{thm-main} could hold true.
\end{enumerate}
\end{remark}

\subsection{Reduction to a statement in complex bialgebraic geometry} Let $X$ be a curve as in \S \ref{sec-sett}. We will first prove that, if there are infinitely many bialgebraic curves in $Y(\R)$, then $\Gamma$ is arithmetic.

\subsubsection{The complex conjugate curve} Let $\tau: \Spec(\C)\rightarrow \Spec(\C)$ be the morphism induced by complex conjugation, and let ${^\tau X}$ be the fiber product
\begin{center}
\begin{tikzcd}
{}^{\tau} X \arrow{r}{} \arrow[swap]{d}{} & X \arrow{d}{} \\%
\Spec (\C) \arrow{r}{\tau}& \Spec(\C).
\end{tikzcd}
\end{center}
For every $\C$-point $\Spec(\C)\rightarrow X$, precomposition with $\tau$ yields a $\C$-point of $^\tau X$, hence we obtain a map $\iota_X: X(\C) \rightarrow {^\tau X}(\C)$.

Concretely, if $X$ is proper, we can fix a closed embedding $X \hookrightarrow \Pr^n_\C$, realizing $X$ as the vanishing locus of finitely many homogeneous complex polynomials $P_1, \ldots, P_k$. The curve $^\tau X$ is the vanishing locus of $\overline{P}_1, \ldots, \overline{P}_k$, where $\overline{P}_i$ is the polynomial whose coefficients are the complex conjugates of those of $P_i$. The map $\iota_X$ sends $[x_0: \cdots : x_n]$ to $[\overline{x}_0: \cdots : \overline{x}_n]$. If $X$ is not proper then it is affine; fixing a closed embedding $X \hookrightarrow \A^n_\C$, we obtain a similar description of $^\tau X$, with arbitrary polynomials instead of homogeneous ones. The above formula shows that $\iota_X$ is continuous, hence a homeomorphism.

Let $\iota: \Hp \rightarrow \Hp$ be the map sending $z$ to $-\bar{z}$, and let $^\tau p=\iota_X \circ p \circ \iota: \Hp \rightarrow {^\tau X}(\C)$. By construction, the following diagram commutes:
\begin{center}
\begin{tikzcd}
\Hp \arrow{r}{\iota} \arrow[swap]{d}{p} & \Hp \arrow{d}{{^\tau p}} \\%
X (\C) \arrow{r}{\iota_X}& {}^{\tau} X(\C).
\end{tikzcd}
\end{center}

Let $\mathsf{C}=\begin{pmatrix}
1 & 0\\ 0 & -1
\end{pmatrix}$, and $^\tau \Gamma = \mathsf{C} \Gamma \mathsf{C}$.

\begin{lemma}
The map $^\tau p$ is holomorphic, and it is a Galois cover with Galois group $^\tau \Gamma$.
\end{lemma}

\begin{proof}
If $X$ is proper, embed it in $\Pr^n$; the map $p: \Hp \rightarrow X(\C)\subset \Pr^n(\C)$ is given by $n+1$ holomorphic functions $(p_0, \ldots, p_n)$ on $\Hp$, hence $^\tau p$ is given by the functions $^\tau p_i(z)=\overline{p_i(-\bar{z})}$, which are holomorphic. A similar argument works if $X$ is not proper, embedding it in an affine space.

The map $^\tau p$ is a cover because $p$ is a cover and the maps $\iota_X$ and $\iota$ are homeomorphisms. Finally, the fact that $^\tau p$ is Galois with group $^\tau \Gamma$ follows from a direct computation, using the fact that, for $\mathsf{M} \in \mathrm{GL}_{2}(\R)$ with positive determinant and $z \in \Hp$, we have $\iota \mathsf{M} \iota(z)=\mathsf{C M C} (z)$.
\end{proof}

\subsubsection{The base change diagram}\label{ss:bcdiag} The base change $Y_\C=Y \times_{\Spec(\R)}\Spec(\C)$ is isomorphic to $X \times_{\Spec(\C)} {^\tau X}$, and the composite of the maps
\begin{equation*}
X(\C)=Y(\R)\hookrightarrow Y(\C)\xrightarrow{\sim} X(\C)\times {^\tau X(\C)}
\end{equation*}
is given by $(\mathrm{Id}, \iota_X)$. Hence, the following diagram commutes:
\begin{center}
\begin{tikzcd}
\Hp \arrow{r}{ (\mathrm{Id}, \iota)} \arrow[swap]{d}{p} & \Hp \times \Hp \arrow{d}{p \times ^\tau p} \\ 
X (\C) =Y(\R) \arrow{r}{} & Y(\C)= X(\C)  \times {}^{\tau} X(\C).
\end{tikzcd}
\end{center}

We will make use of the map $f: \C^2 \rightarrow \C^2$ sending $(x, y)$ to $(x+iy, -x+iy)$.

\begin{definition}
We call a subset $\mathcal{Z}\subset X(\C)\times {^\tau X}(\C)$ a complex bialgebraic curve if it satisfies the following properties:
\begin{enumerate}
\item it is the set of complex points of an irreducible complex algebraic curve $Z \subset X \times {^\tau X}$;
\item there is an irreducible complex algebraic curve $\hat{Z} \subset \Pr^1_\C \times \Pr^1_\C$ and a non-empty open subset (in the Euclidean topology) $U \subset \hat{Z}(\C) \cap (\Hp \times \Hp)$ such that $(p \times {^\tau p})(U)\subset \mathcal{Z}$.
\end{enumerate}
\end{definition}

\begin{lemma}\label{lem-realcpxbialg}
If $\mathcal{C}=C(\R) \subset Y(\R)$ is a bialgebraic curve then $C(\C)\subset Y(\C)=X(\C) \times {^\tau X}(\C)$ is a complex bialgebraic curve.
\end{lemma}
\begin{proof}
Take a curve $\hat{C}\subset \A^2_\R$ as in Remark \ref{rem-alginH}(1), with an arc $\mathcal{I}\subset \hat{\mathcal{C}}=C(\R)\cap \Hp$ such that $p(\mathcal{I})\subset \mathcal{C}$. We have the following commutative diagram:

\begin{center}
\begin{tikzcd}
\R^2  \arrow[hookrightarrow, r]  &[0.5em] \C^2 \arrow[r, "f"] & \C^2 \\
\mathcal{I}\subset \hat{\mathcal C} \arrow[hookrightarrow, u] \arrow[hookrightarrow, r] & \Hp \arrow{r}{ (\mathrm{Id}, \iota)} \arrow[hookrightarrow, ul] \arrow[d, "p"] & \Hp \times \Hp \arrow[hookrightarrow, u] \arrow[d, "p \times {^\tau p}"] \\
& Y(\R) \arrow[r, ] & X(\C) \times {^\tau X}(\C) \\
& \mathcal C =C(\R) \arrow[hookrightarrow, u] \arrow[r] & C(\C)  \arrow[hookrightarrow, u].
\end{tikzcd}
\end{center}

Let $\hat{C}_\C\subset \A^2_\C$ be the base change of $\hat{C}$, and let $C_\C\subset Y_\C$ be the base change of $C$.
Note that $\hat{C}$ and $C$ are irreducible and contain a dense set of real points, hence $\hat{C}_\C$ and $C_\C$ are irreducible. 
We claim that $(p \times {^\tau p})(U)\subset C(\C)$ for some non-empty open subset $U\subset f(\hat{C}(\C))\cap (\Hp \times \Hp)$. Admitting this for the moment, let us view $f$ as an automorphism of $\A^2_\C$; taking $\hat{Z}\subset \Pr^1_\C \times \Pr^1_\C$ to be the Zariski closure of $f(\hat{C}_\C)$ we see that $C(\C)$ is bialgebraic.

Let us now prove our claim. The arc $\mathcal{I}$ is a neighbourhood of a smooth point $x$ of $\hat{C}$, and the image of $\mathcal{I}$ via $(\mathrm{Id}, \iota)$ is contained the intersection $f(\hat{C}(\C))\cap (p \times {^\tau p})^{-1}(C(\C))$. On the other hand, the intersection $f(\hat{C}(\C))\cap (\Hp \times \Hp)$ is an open neighbourhood of $f(x)$ in $f(\hat{C}(\C))$, hence it contains an open neighborhood $U$ of $f(x)$ biholomorphic to a disc. Furthermore, the set $(p \times {^\tau p})^{-1}(C(\C)) \cap U$ is a complex analytic closed, uncountable subset of $U$. Therefore, we have $U \subset (p \times {^\tau p})^{-1}(C(\C))$.
\end{proof}

\subsubsection{} Let $\BA_Y$ be the set of bialgebraic curves in $Y(\R)$, and let $\BA'_{X \times ^\tau X}$ be the set of bialgebraic curves in $X(\C)\times {^\tau X}(\C)$ that are neither horizontal nor vertical. Note that, if $\mathcal{C}=C(\R)$ belongs to $\BA_Y$, then $C(\C)$ belongs to $\BA'_{X \times {^\tau X}}$. Indeed, by the previous lemma $C(\C)$ is bialgebraic; furthermore, no coordinate can be constant on $C(\C)\subset X(\C)\times {^\tau X}(\C)$, as $p(\mathcal{I})$ is contained in $C(\R)$ and it is not a singleton. Finally, note that the map $\BA_Y \rightarrow \BA'_{X \times ^\tau X}$ sending $\mathcal{C}$ to $C(\C)$ is injective. Indeed, if $C_1(\R)=\mathcal{C}_1 \neq \mathcal{C}_2=C_2(\R),$ then $C_1$ and $C_2$ are distinct irreducible curves in $Y$, hence $C_1(\C)\neq C_2(\C)$.

Therefore, the fact that $\BA_Y$ is finite if $\Gamma$ is not arithmetic, which is one of the two implications in Theorem \ref{thm-main}, follows from the next result.

\begin{theorem}\label{thm-klin}
If $\Gamma \subset \PGL(\R)^+$ is not arithmetic, then the set $\BA'_{X \times {^\tau X}}$ is finite.
\end{theorem}

\begin{remark}
The proof of Theorem \ref{thm-klin} that we will give below follows a strategy that was explained to us by Bruno Klingler. It rests on a relation between bialgebraic curves in $X(\C)\times {^\tau X(\C)}$ that are neither horizontal nor vertical and elements of the commensurator of $\Gamma$ in $\PGL(\R)$, and on Margulis' theorem \cite[Theorem 1']{mar91}.
    
Let us also mention that stronger functional transcendence results are known for uniformization maps of products of a genus zero quotient of $\Hp$ by a Fuchsian group \cite{cfn20}.
\end{remark}

\subsection{Description of the curves in $\BA'_{X \times {^\tau X}}$}\label{sec-descbialg}
Let $X$ be a hyperbolic algebraic curve as in \S \ref{sec-sett} (we do not assume that $\Gamma$ is arithmetic). We will now describe bialgebraic curves in $X(\C)\times {^\tau X(\C)}$ that are neither horizontal nor vertical.

\subsubsection{} Let $\pi=p \times {^\tau p}: \Hp \times \Hp \rightarrow X \times {^\tau X}$. Let $i: Z \hookrightarrow X \times {^\tau X}$ be an irreducible curve, and let $\hat{Z}\subset \Pr^1_\C \times \Pr^1_\C$ be an irreducible curve such that there is a non-empty open subset (in the Euclidean topology) $U \subset \hat{Z}\cap (\Hp \times \Hp)$ satisfying $\pi(U)\subset Z$. Let $q: \tilde{Z}\rightarrow Z$ be the normalization of $Z$. Note that $\tilde{Z}$ is irreducible, hence $\tilde{Z}(\C)$ is connected in the Euclidean topology. The universal cover of $\tilde{Z}$ admits a non-constant holomorphic map to $\Hp \times \Hp$, hence it is biholomorphic to $\Hp$; we fix a covering map $\tilde{\pi}: \Hp \rightarrow \tilde{Z}$, and choose a (holomorphic) map $\ell: \Hp \rightarrow \Hp \times \Hp$ such that $i \circ q \circ \tilde{\pi}=\pi \circ \ell$. To sum up, we have the following commutative diagram:

\begin{center}
\begin{tikzcd}
& & U \arrow[hookrightarrow, d] \arrow[hookrightarrow, r] & \hat{Z} \arrow[hookrightarrow, d] \\
\Hp \arrow[d, "\tilde{\pi}"] \arrow[rr, "\ell"] & & \Hp \times \Hp \arrow[d, "\pi"] \arrow[hookrightarrow, r] & \Pr^1_\C \times \Pr^1_\C\\
\tilde{Z} \arrow[r, "q"] & Z \arrow[hookrightarrow, r, "i"] & X \times {^\tau X}. &
\end{tikzcd}
\end{center}

\begin{proposition}\label{mainprop}
The following assertions hold true.
\begin{enumerate}
\item There exists $h \in \PGL(\R)^+$ such that $\hat{Z}(\C)=\{(x, hx), x \in \Pr^1(\C)\}$, and $\mathsf{C}h$ belongs to the commensurator of $\Gamma$ in $\PGL(\R)$.
\item We have $\pi(\hat{Z}(\C)\cap (\Hp \times \Hp))=i(Z(\C))$.
\end{enumerate}
\end{proposition}

\subsection{Proof of Proposition \ref{mainprop}} The notation of \S \ref{sec-descbialg} is in force. The key point in the proof is to show that the stabilizer of $\hat{Z}$ in $\mathrm{Aut}(\Pr^1_\C)\times\mathrm{Aut}(\Pr^1_\C)$ is ``large'' - precisely, isomorphic to $\mathrm{PGL}_{2, \C}$ - exploiting the action of the fundamental group of $\tilde{Z}$ on $\ell(\Hp)$.

\subsubsection{Topology} Fix $z_0 \in \Hp$ and let $\tilde{z}_0=\tilde{\pi}(z_0)$; let $\tilde{\Gamma}=\pi_1(\tilde{Z}, \tilde{z}_0)$. For every $\gamma \in \tilde{\Gamma}$, let $\tilde{\gamma}: [0, 1] \rightarrow \Hp$ be the lift of $\gamma$ starting at $z_0$. There is a unique map $\varphi_{\gamma}: \Hp \rightarrow \Hp$ such that $\tilde{\pi}\circ \varphi_\gamma=\tilde{\pi}$ and $\varphi_\gamma(z_0)=\tilde{\gamma}(1)$. Sending $\gamma \in \tilde{\Gamma}$ to $\varphi_{\gamma}$ we obtain an action of $\tilde{\Gamma}$ on $\Hp$ (cf. the proof of \cite[Proposition 1.39, p. 71]{hat02}). Similarly, we have an action of $\pi_1(X \times {^\tau X}, i \circ q(\tilde{z}_0))$ on $\Hp \times \Hp$, defined using $\ell(z_0)$ as a base point of $\Hp \times \Hp$. For every $\gamma \in \tilde{\Gamma}$, the maps $\ell \circ \varphi_\gamma$ and $\varphi_{(i \circ q)_*(\gamma)}\circ \ell$ are lifts of $\pi \circ \ell$ with the same value at $z_0$, hence they coincide. Therefore, for every $z \in \Hp$ we have
\begin{equation*}
\ell(\gamma \cdot z)=(i \circ q)_*(\gamma)\cdot (\ell(z)).
\end{equation*}
In other words, letting $\tilde{\Gamma}$ act on $\Hp \times \Hp$ via the morphism $(i\circ q)_*: \tilde{\Gamma}\rightarrow \pi_1(X \times {^\tau X}, i \circ q(\tilde{z}_0))$ and the action of the target on $\Hp \times \Hp$, the map $\ell$ is $\tilde{\Gamma}$-equivariant. In particular, we have
\begin{equation}\label{eq-gammatl}
\tilde{\Gamma}\cdot \ell(\Hp)\subset \ell(\Hp).
\end{equation}

The above action yields an injective morphism $\pi_1(X \times {^\tau X}, i \circ q(\tilde{z}_0)) \rightarrow \mathrm{Aut}(\Hp) \times \mathrm{Aut}(\Hp)\simeq  \PGL(\R)^+\times \PGL(\R)^+$ with image $\Gamma \times {^\tau \Gamma}$; the action of $\PGL(\R)^+\times \PGL(\R)^+$ on $\Hp \times \Hp$ is the restriction of the natural action of $\PGL(\C)\times \PGL(\C)$ on $\Pr^1_\C \times \Pr^1_\C$. We denote the composite
\begin{equation*}
\tilde{\Gamma} \xrightarrow{(i \circ q)_*} \pi_1(X \times {^\tau X}, i \circ q(\tilde{z}_0)) \rightarrow \mathrm{Aut}(\Hp) \times \mathrm{Aut}(\Hp)\simeq  \PGL(\R)^+\times \PGL(\R)^+
\end{equation*}
by $j$.

\begin{lemma}\label{lem:gammastabhatz}
There exists $\gamma \in \Gamma \times {^\tau \Gamma}$ such that $\gamma \cdot \ell(\Hp)\subset \hat{Z} \cap (\Hp \times \Hp)$.
\end{lemma}
\begin{proof}
For every $\gamma \in \Gamma \times {^\tau \Gamma}$, the action of $\gamma^{-1}$ yields a bijection between $\hat{Z}\cap (\gamma \cdot \ell(\Hp))$ and $(\gamma^{-1} \cdot \hat{Z})\cap \ell(\Hp)$. We need to show that there is $\gamma \in \Gamma \times {^\tau \Gamma}$ such that $(\gamma^{-1} \cdot \hat{Z})\cap \ell(\Hp)=\ell(\Hp)$. Suppose by contradiction that $(\gamma^{-1} \cdot \hat{Z})\cap \ell(\Hp)\subsetneq \ell(\Hp)$ for every $\gamma \in \Gamma \times {^\tau \Gamma}$; then, for each $\gamma$ the set $A_\gamma=\ell^{-1}((\gamma^{-1} \cdot \hat{Z})\cap (\Hp \times \Hp))$ is properly contained in $\Hp$. The set $A_\gamma$ is the preimage of $\gamma^{-1} \cdot \hat{Z}$ via the composite of $\ell$ and the inclusion $\Hp \times \Hp \rightarrow \Pr^1_\C \times \Pr^1_\C$; therefore, it is a closed analytic subset of $\Hp$. As it is properly contained in $\Hp$, it must be countable. Since this is true for every $\gamma \in \Gamma \times {^\tau \Gamma}$, the union $\cup_{\gamma \in \Gamma \times {^\tau \Gamma}} A_\gamma$ is countable. But the image via $\ell$ of $\cup_{\gamma \in \Gamma \times {^\tau \Gamma}} A_\gamma$ is $\cup_{\gamma \in \Gamma \times {^\tau \Gamma}} (\gamma^{-1}\cdot \hat{Z} \cap \ell(\Hp))$, which is in bijection with $\cup_{\gamma \in \Gamma \times {^\tau \Gamma}} (\hat{Z} \cap \gamma \cdot \ell(\Hp))$. As $\pi(\ell(\Hp))=i(Z)$, the set $\cup_{\gamma \in \Gamma \times {^\tau \Gamma}} (\hat{Z} \cap \gamma \cdot \ell(\Hp))$ is equal to $\hat{Z}\cap \pi^{-1}(i(Z))$, which contains $U$ by assumption, hence it is uncountable.
\end{proof}

\subsubsection{}\label{sec:gammastabhatz} Fix $\gamma \in \Gamma \times {^\tau \Gamma}$ such that $\gamma \cdot \ell(\Hp)\subset \hat{Z} \cap (\Hp \times \Hp)$. By \eqref{eq-gammatl} we have $(\gamma j(\tilde{\Gamma}) \gamma^{-1})\cdot(\gamma \cdot \ell(\Hp))\subset \gamma \cdot \ell(\Hp)$; therefore, for every $\gamma' \in \gamma j(\tilde{\Gamma}) \gamma^{-1}$, we have $\gamma'\cdot (\gamma \cdot \ell(\Hp))\subset \hat{Z} \cap \gamma'\cdot \hat{Z}$. As a consequence, the two curves $\hat{Z}$ and $\gamma' \cdot \hat{Z}$ are equal.

We have proved that the action of $\gamma j(\tilde{\Gamma})\gamma^{-1}$ sends $\hat{Z}$ to itself; hence, the action of $\tilde{\Gamma}$ sends $\gamma^{-1}\cdot \hat{Z}$ to itself.

\subsubsection{Algebraic group theory}\label{ssec:algp} Let $G \subset \mathrm{Aut}(\Pr^1_\C)\times \mathrm{Aut}(\Pr^1_\C)\simeq \mathrm{PGL}_{2, \C}^2$ be the stabilizer of $\gamma^{-1}\cdot \hat{Z}$. It is a closed algebraic subgroup of $\mathrm{PGL}_{2, \C}^2$ \cite[Theorem 2.2.6]{bri17}; by the previous paragraph, the complex points of $G$ contain $j(\tilde{\Gamma})$. As the map $q$ is surjective and the curve $Z$ is neither horizontal nor vertical, the projection of $j(\tilde{\Gamma})$ on the first (resp. second) factor is a finite index subgroup of $\Gamma$ (resp. $^\tau \Gamma$). Indeed, let $q_1: \tilde{Z}\rightarrow X$ be the composite of the map $i \circ q: \tilde{Z}\rightarrow X \times {^\tau X}$ and of the first projection. There is a non-empty open $V \subset X$ such that $q_1: q_1^{-1}(V) \rightarrow V$ is finite étale (use \cite[\href{https://stacks.math.columbia.edu/tag/0BAI}{Tag 0BAI}]{stacks-project}, then remove the ramfication locus); as the fundamental group of $V$ surjects on the fundamental group of $X$, the image of $\tilde{\Gamma}$ in $\Gamma$ has finite index. The argument for the image of $\tilde{\Gamma}$ in $^\tau\Gamma$ is similar.

Therefore, the projection of $G \subset \mathrm{PGL}_{2, \C}^2$ on each factor is a closed subgroup scheme of $\mathrm{PGL}_{2, \C}$ \cite[Proposition 2.7.1]{bri17} whose complex points contain a lattice in $\mathrm{PGL}_{2}(\R)^+$. As lattices in $\PGL(\R)^+$ are Zariski dense in $\mathrm{PGL}_{2, \C}$ (by Borel's density theorem \cite[Corollary 4.5.6]{mor15}), the projection of $G \subset \mathrm{PGL}_{2, \C}^2$ on each factor is surjective.

\begin{lemma}\label{lem:descG}
The composite $\varphi$ of the inclusion $G\rightarrow \mathrm{PGL}_{2, \C}^2$ and of the projection on the first factor is an isomorphism. Furthermore, the map $G \rightarrow \mathrm{PGL}_{2, \C}^2$ is given by $(\varphi, h \varphi h^{-1})$ for some $h\in \PGL(\C)$.
\end{lemma}
\begin{proof}
Let $p_1$ and $p_2$ be the projections of $\mathrm{PGL}_{2, \C}^2$ on the two factors, and let $K=\mathrm{ker}(p_{1|G})$. Using that $p_{2|G}$ is surjective one checks that the subgroup $K(\C)\subset \{1\} \times \PGL(\C)$ is normal. It follows that $K=\{1\}$: indeed, as $\PGL(\C)$ is simple, it suffices to show that $K(\C)\neq \PGL(\C)$. Assume by contradiction that $K(\C)=\PGL(\C)$. The same argument with $p_2$ instead of $p_1$ shows that $G(\C) \cap (\PGL(\C)\times \{1\})$ is either $\{1\}$ or $\PGL(\C)\times \{1\}$. In the first (resp. second) case we deduce that $G(\C)=\{1\}\times \PGL(\C)$ (resp. $G(\C)=\PGL(\C)\times \PGL(\C)$), contradicting the fact that $G$ stabilizes a non-vertical curve.

We have shown that $p_1$ is injective when restricted to $G(\C)$; since we already know that it is surjective, we deduce that the restriction $\varphi$ of $p_1$ to $G$ is an isomorphism. The same argument with $p_2$ instead of $p_1$ shows that the restriction $\varphi'$ of $p_2$ to $G$ is an isomorphism. The composite $\varphi'\circ \varphi^{-1}$ is an automorphism of $\mathrm{PGL}_{2, \C}$, which is inner by \cite[Proposition 25.15]{kmrt98}. This completes the proof.
\end{proof}

\subsubsection{Description of $\hat{Z}$}\label{ssec:deschatz} By the above discussion the curve $\gamma^{-1}\hat{Z}\subset \Pr^1_\C \times \Pr^1_\C$ is stabilized by a group of the form $\{(g, hgh^{-1}), g \in \PGL(\C)\}$, for some $h \in \PGL(\C)$. Hence, the curve $\hat{Z}$ is stabilized by a group of this form as well. From now on, we fix $h \in \PGL(\C)$ such that the stabilizer of $\hat{Z}$ has complex points $\{(g, hgh^{-1}), g \in \PGL(\C)\}$.
\begin{lemma}\label{lem:hatz}
We have $\hat{Z}(\C)=\{(x, hx), x \in \Pr^1(\C)\}$. Furthermore, the element $h$ belongs to $ \PGL(\R)^+$.
\end{lemma}
\begin{proof}
The map $\Pr^1_\C\times \Pr^1_\C\rightarrow \Pr^1_\C\times \Pr^1_\C$ sending $(x, y)$ to $(x, h^{-1}y)$ is equivariant with respect to the $\PGL(\C)$-action on the source (resp. target) induced by the embedding  $\PGL(\C)\rightarrow \PGL(\C)^2$ sending $g$ to $(g, hgh^{-1})$ (resp. to $(g, g)$). Hence, to prove the first assertion we may assume that $h=1$, and show that $\hat{Z}$ is the image of the diagonal embedding $\Delta: \Pr^1_\C \rightarrow (\Pr^1_\C)^2$. The action of $\PGL(\C)$ on couples of distinct points of $\Pr^1(\C)$ is transitive. Therefore, if $\hat{Z}(\C)$ contains a point outside the diagonal, it contains the whole complement of the diagonal, contradicting the fact that $\hat{Z}$ is a curve. It follows that the curve $\hat{Z}$ is contained in, hence equal to, the image of $\Delta$.

Let us prove that $h$ belongs to $\PGL(\R)^+$. Recall that $\gamma j(\tilde{\Gamma}) \gamma^{-1}\subset \PGL(\R)^2$ stabilizes $\hat{Z}$. Hence, every element $(g, hgh^{-1}) \in \PGL(\C)^2$ belonging to $\gamma j(\tilde{\Gamma}) \gamma^{-1}$ is an element of the stabilizer of $\hat{Z}$ fixed by complex conjugation; therefore, for such an element conjugation by $\bar{h}^{-1}h$ fixes $g$. In other words, conjugation by the element $\bar{h}^{-1}h$ fixes a conjugate of the image of $\tilde{\Gamma}$ in the first factor of $\PGL(\C)^2$, hence it is equal to the identity (because the image of $\tilde{\Gamma}$ in $\PGL(\C)$ is Zariski dense). Therefore, the element $h$ belongs to $\PGL(\R)$; finally, as $\hat{Z}$ has non-empty intersection with $\Hp\times \Hp$, we must have $\det(h)>0$.
\end{proof}

\subsubsection{The commensurator of $\Gamma$}\label{ssec-comm} Consider the commensurator of $\Gamma$ in $\PGL(\R)$:
\begin{equation*}
\mathrm{Comm}_{\PGL(\R)}(\Gamma)=\{\mathsf{M} \in \PGL(\R) \mid \mathsf{M}^{-1} \Gamma \mathsf{M} \cap \Gamma \text{ has finite index in } \Gamma \text{ and } \mathsf{M}^{-1}\Gamma \mathsf{M} \};
\end{equation*}
it is a subgroup of $\PGL(\R)$ containing $\Gamma$.

For $h$ as in \S \ref{ssec:deschatz}, we will show that $\mathsf{C}h$ belongs to $\mathrm{Comm}_{\PGL(\R)}(\Gamma)$. Denoting the coordinates of $\gamma j\gamma^{-1}: \tilde{\Gamma} \rightarrow \PGL(\R)^2$ by $j_1$ and $j_2$, we have $j_2=h j_1 h^{-1}$; furthermore, we deduce from \S \ref{ssec:algp} that $j_1(\tilde{\Gamma})$ has finite index in $\Gamma$ and $h j_1(\tilde{\Gamma})h^{-1}$ has finite index in $^\tau \Gamma=\mathsf{C}\Gamma\mathsf{C}$. Therefore $j_1(\tilde{\Gamma})$ has finite index in $(\mathsf{C} h)^{-1}\Gamma\mathsf{C} h$; it follows that $\mathsf{C}h$ belongs to the commensurator of $\Gamma$ in $\PGL(\R)$. This completes the proof of Proposition \ref{mainprop}(1). The next lemma yields the second statement in the proposition.

\begin{lemma}\label{lem:descz}
We have $\pi(\hat{Z}\cap (\Hp \times \Hp))=i(Z)$.
\end{lemma}
\begin{proof}
    The intersection $\hat{Z} \cap \pi^{-1}(i(Z))$ is a closed analytic subset of $\hat{Z} \cap (\Hp \times \Hp) \simeq \Hp$ containing $U$; therefore, we have $\pi(\hat{Z}\cap(\Hp \times \Hp))\subset i(Z)$. To show the other inclusion, let $X'= h^{-1} {^\tau\Gamma} h \cap \Gamma \backslash\Hp$ and $X''= h\Gamma h^{-1} \cap {^\tau\Gamma} \backslash\Hp$. The action of $h$ on $\Hp$ induces a map $m_h: X' \rightarrow X''$, and using Lemma \ref{lem:hatz} one checks that $\pi(\hat{Z}\cap (\Hp \times \Hp))$ is the image in $X \times {^\tau X}$ of the graph of $m_h$. By the discussion in \S \ref{ssec-comm}, the group $h^{-1} {^\tau\Gamma} h \cap \Gamma$ has finite index in $\Gamma$, and $h\Gamma h^{-1} \cap {^\tau\Gamma}$ has finite index in $^\tau \Gamma$. Therefore the projections $X' \rightarrow X$ and $X'' \rightarrow {^\tau X}$ are finite covers and, by (a version of) Riemann's existence theorem \cite[Exposé XII, Théorème 5.1]{gro03}, the curves $X'$ and $X''$ are algebraic. In addition, by \cite[Theorem 1.2(ii)]{jk20} the morphism $m_h$ is algebraic, hence the graph of $m_h$ is a one-dimensional Zariski closed subset of $X' \times X''$.
It follows that the image in $X \times {^\tau X}$ of the the graph of $m_h$ is a Zariski closed one-dimensional subset. As $Z$ is an irreducible curve, we obtain that $\pi(\hat{Z}\cap(\Hp \times \Hp)) = i(Z)$.
\end{proof}

\subsection{Proof of Theorem \ref{thm-klin}}\label{sec:proofmainth}

In the following statement, given $h \in \PGL(\R)^+$, we denote by $\hat{Z}_h\subset \Pr^1_\C\times \Pr^1_\C$ the graph of the morphism given by the action of $h$.

\begin{lemma}\label{lem:z12}
Let $Z_1$ and $Z_2$ be two bialgebraic curves in $X \times {^\tau X}$. Let $h_1$ and $h_2$ be two elements of $\PGL(\R)^+$ with the following properties:
\begin{enumerate}
\item $\mathsf{C} h_1$ and $\mathsf{C} h_2$ belong to $\mathrm{Comm}_{\PGL(\R)}(\Gamma)$;
\item $\pi(\hat{Z}_{h_j}\cap (\Hp \times \Hp))=i(Z_j)$ for $j \in \{1, 2\}$.
\end{enumerate}
If $\mathsf{C} h_1$ and $\mathsf{C} h_2$ have the same image in $\Gamma \backslash \mathrm{Comm}_{\PGL(\R)}(\Gamma)$ then $i(Z_1)=i(Z_2)$.
\end{lemma}
\begin{proof}
Assume that $\mathsf{C}h_2=\gamma \mathsf{C} h_1$ for some $\gamma \in \Gamma$. Sending $(z, z')$ to $(z, \mathsf{C}\gamma \mathsf{C} z')$ yields a bijection
\begin{equation*}
\{(z, z') \in \Hp \times \Hp \mid z'=h_1 z \}\rightarrow \{(z, z') \in \Hp \times \Hp \mid z'=h_2 z\}.
\end{equation*}
Furthermore, as $\mathsf{C} \Gamma \mathsf{C}= {^\tau \Gamma}$, the points $(z, z')$ and $(z, \mathsf{C}\gamma \mathsf{C} z')$ are in the same $\Gamma \times {^\tau \Gamma}$-orbit for every $(z, z') \in \Hp \times \Hp$. Therefore, we have $\pi(\hat{Z}_{h_2}\cap (\Hp \times \Hp))=\pi(\hat{Z}_{h_1}\cap (\Hp \times \Hp))$.
\end{proof}

\subsubsection{Conclusion of the proof of Theorem \ref{thm-klin}} Assume that $\BA'_{X \times {^\tau X}}$ is infinite; we need to show that $\Gamma$ is arithmetic. For each $Z \in \BA'_{X \times {^\tau X}}$, choose - thanks to Proposition \ref{mainprop} - an element $h \in \PGL(\R)^+$ such that $\mathsf{C}h$ belongs to $\mathrm{Comm}_{\PGL(\R)}(\Gamma)$ and $\pi(\hat{Z}_{h}\cap (\Hp \times \Hp))=i(Z)$. Lemma \ref{lem:z12} implies that $\Gamma$ has infinite index in $\mathrm{Comm}_{\PGL(\R)}(\Gamma)$; hence, the index of $\Gamma$ in $\mathrm{Comm}_{\PGL(\R)^+}(\Gamma)$ is infinite as well. By Margulis' theorem \cite[Theorem 1']{mar91} the lattice $\Gamma$ is arithmetic.

\subsubsection{Bialgebraic curves and antiholomorphic involutions}\label{ssec: bialisgeo} The above discussion allows us to describe, in the setting of \S \ref{sec-sett}, bialgebraic curves in $Y(\R)$ in terms of (the real points of) étale covers of $X$ defined over the reals, extending Remark \ref{rem-bialgrealell} beyond projective curves of genus one.

Let $\mathcal{C}=C(\R)\subset Y(\R)$ be a bialgebraic curve, and let $\hat{\mathcal{C}}=\hat{C}(\R)\cap \Hp$ be an algebraic curve as in Remark \ref{rem-alginH}(1), with an arc $\mathcal{I}\subset \hat{\mathcal{C}}$ such that $p(\mathcal{I})\subset \mathcal{C}$. The proof of Lemma \ref{lem-realcpxbialg}, together with Proposition \ref{mainprop}, shows that the Zariski closure $\hat{Z}$ of $f(\hat{C}_\C)$ in $\Pr^1_\C \times \Pr^1_\C$ has complex points
\begin{equation*}
\hat{Z}(\C)=\{(x, hx), x \in \Pr^1(\C)\}
\end{equation*}
for some $h \in \PGL(\R)^+$ such that $\mathsf{C}h$ belongs to the commensurator of $\Gamma$ in $\PGL(\R)$. Therefore, choosing a matrix $\mathsf{A} \in \mathrm{GL}_2(\R)$ with image $\mathsf{C}h \in \PGL(\R)$, we obtain that $\hat{\mathcal{C}}=\{z \in \Hp \mid \mathsf{A}z=\bar{z}\}$.
This forces the trace of $\mathsf{A}$ to be 0, hence $\hat{\mathcal{C}}$ is a geodesic in $\Hp$ (cf. \cite[\S 3.3]{Tam23}). We remark that $p(\hat{\mathcal{C}})$ is contained in $\mathcal{C}$ (either because of the commutativity of the diagram in \S \ref{ss:bcdiag} and Lemma \ref{lem:descz}, or by analytic continuation).

Note that $\mathsf{A}$ is equal to its inverse in $\PGL(\R)$. Letting $X'=\mathsf{A} \Gamma \mathsf{A} \cap \Gamma \backslash \Hp$, the map $z \mapsto \mathsf{A} \bar{z}$ induces an anti-holomorphic involution $\iota_{\mathsf{A}}$ of the Riemann surface $X'$, which is algebraic: this is shown as in \cite[\S 3]{hui00} (the existence of the morphism $f$ in \cite[\S 3, p. 151]{hui00} being a consequence of \cite[Theorem 1.2(ii)]{jk20}). Therefore, the involution $\iota_{\mathsf{A}}$ yields an effective descent datum for the complex algebraic curve $X'$ to a real algebraic curve $C'$ \cite[Corollary 7.3]{mil24}, and the image of $\hat{\mathcal{C}}$ via the quotient map $\Hp \rightarrow X'$ is contained in $C'(\R)$. Therefore, letting $Y'=\mathrm{Res}_{\C/\R}X'$, the set $C'(\R)$ is a real bialgebraic curve in $Y'(\R)$ (see also Example \ref{ex-realbalg}). As the projection $q: Y'\rightarrow Y$ is finite, the image $q(C')$ is an irreducible algebraic curve, hence it coincides with $C$ (as both curves contain $p(\hat{\mathcal{C}})$).

To sum up, we have achieved the following description of bialgebraic curves in $Y(\R)$.

\begin{proposition}\label{prop-bialfromrealcover}
Let $\mathcal{C}\subset Y(\R)$ be a bialgebraic curve. There is a smooth, geometrically connected real algebraic curve $C'$ with non-empty set of real points and a finite étale morphism $X'=C'\times_{\Spec(\R)}\Spec(\C) \rightarrow X$ with restriction of scalars $q: Y' \rightarrow Y$ such that $\mathcal{C}$ is the set of real points of $q(C')$.
\end{proposition}

\subsection{Bialgebraic curves in the arithmetic case}\label{ssec:bialgarit}

\subsubsection{} Fix a hyperbolic algebraic curve $X=\Gamma \backslash\Hp$, and suppose now that $\Gamma \subset \PGL(\R)^+$ is arithmetic; we will give two descriptions of the elements of $\BA_Y$; in particular we will show that this set is infinite, completing the proof of Theorem \ref{thm-main}. For this purpose, we may replace $\Gamma$ by a conjugate subgroup, or by a subgroup of finite index: indeed, suppose that $\Gamma'\subset \Gamma$ has finite index and let $q: Y' \rightarrow Y$ be the projection map, where $Y'(\R)=\Gamma'\backslash \Hp$. If $\mathcal{C}'=C'(\R)\subset Y'(\R)$ is a bialgebraic curve and $C=q(C')$ then $\mathcal{C}=C(\R)$ is a bialgebraic curve, and the resulting map $\BA_{Y'}\rightarrow \BA_Y$ has finite fibers.

\subsubsection{The non cocompact case} If $\Gamma\backslash \Hp$ is not compact then, up to conjugation, the group $\Gamma$ is commensurable with $\mathrm{SL}_2(\Z)/\{\pm \mathrm{Id}\}$ \cite[Proposition 6.1.5]{mor15}. Hence we may assume that $\Gamma$ is a finite index subgroup of $\mathrm{SL}_2(\Z)/\{\pm \mathrm{Id}\}$, and we have a map $q: Y \rightarrow \A^2_\R$ inducing the quotient map $\Gamma \backslash \Hp \rightarrow \mathrm{SL}_2(\Z)\backslash \Hp$ on real points. If $\mathcal{S} \subset \Hp$ is a special geodesic, as defined in \cite[Definition 3.3.5]{Tam23}, then $p(\mathcal{S}) \subset Y(\R)$ is contained in the preimage $C$ via $q$ of one of the curves $\mathcal Z_N$ defined in \cite[\S 3.5.1]{Tam23}. The image of (an arc of) $\mathcal{S}$ via $p$ is contained in an irreducible component of $C$, whose real points are therefore a bialgebraic subset of $Y(\R)$. Varying $\mathcal{S}$, we see that there are infinitely many bialgebraic curves in $Y(\R)$ (in fact, as special geodesics are dense in $\Hp$, so are bialgebraic curves in $Y(\R)$).

\subsubsection{The cocompact case: bialgebraic curves from tori}\label{ssec:atr} Let us now suppose that $\Gamma$ is cocompact. By \cite[Proposition 6.2.6]{mor15} (see also \cite[\S 2]{moc98}), up to conjugating $\Gamma$ and passing to a subgroup of finite index, we may suppose that there exist a totally real field $F$, a non-split quaternion agebra $B$ over $F$ ramified at all but one infinite place of $F$, and an $\mathcal{O}_F$-order $O \subset B$, with the following property. Let $G=\mathrm{Res}_{F/\Q}B^\times$, let $O^{\times, 1}\subset O^\times$ be the subgroup of norm one elements, and let $\tau: B^\times \rightarrow \PGL(\R)$ be the composite of the inclusion $G(\Q)\subset G(\R)$ and of the projection onto the non-compact factor of $G(\R)/(F \otimes \R)^\times$. The group $\Gamma$ is a subgroup of finite index of $\tau(O^{\times, 1})$.

Let $L=F(\sqrt{D})$ be a non-CM quadratic extension of $F$, with $D \in \mathcal{O}_F$, that can be embedded in $B$. Fixing such an embedding, let $\mathsf{A}=\tau(\sqrt{D})$. Note that, for every archimedean place $\sigma: F \rightarrow \R$ where $B$ is ramified, the tensor product $L \otimes_{F, \sigma} \R$ is isomorphic to $\C$. Hence, denoting by $\sigma_0$ the unique archimedean place of $F$ where $B$ is unramified, we have $L \otimes_{F, \sigma_0} \R\simeq \R \times \R$. In particular $\sigma_0(D)$ is positive, hence (any lift to $\mathrm{GL}_2(\R)$ of) $\mathsf{A}$ has trace 0 and negative determinant. Let $\mathcal{S}_\mathsf{A}=\{z \in \Hp \mid \mathsf{A}z=\bar{z}\}$.
\begin{lemma}\label{lem-geodbalg}
The image of $\mathcal{S}_\mathsf{A}$ in $Y(\R)$ is contained in the set of real points of an irreducible algebraic curve.
\end{lemma}
\begin{proof}
Let $\hat{Z} \subset \Pr^1_\C \times \Pr^1_\C$ be the graph of the action of $h=\mathsf{CA}$. As $\mathsf{A}$ belongs to the commensurator of $\Gamma$, the proof of Lemma \ref{lem:descz} shows that $\pi(\hat{Z}\cap (\Hp \times \Hp))$ is an irreducible algebraic curve $Z\subset X \times {^\tau X}$. We will show that $Z$ is the base change of a real algebraic curve $C \subset Y$. This implies that, for every $z \in \mathcal{S}_\mathsf{A}$, the point $p(z)$ belongs to $Y(\R)\cap Z(\C)=C(\R)$. To show the existence of $C$, let $\bar{Z}$ be the image of $Z$ via the action of complex conjugation on $X \times {^\tau X}=Y \times_{\Spec(\R)}\Spec(\C)$. The image of $p(\mathcal{S}_{\mathsf{A}})$ in $(X \times {^\tau X})(\C)$, being fixed by complex conjugation, is contained in $Z(\C) \cap \bar{Z}(\C)$. Therefore we have $Z=\bar{Z}$, hence the curve $Z$ is the base change of a real algebraic curve $C \subset Y$  \cite[Proposition 4.3]{mil24}.
\end{proof}

\subsubsection{} Using the above lemma, let us show that $\BA_Y$ is infinite. Take two different fields $L_1=F(\sqrt{D_1})$ and $L_2=F(\sqrt{D_2})$ as in \S \ref{ssec:atr} embedded in $B$. For $i \in \{1, 2\}$, let $\mathsf{A}_i=\tau(\sqrt{D}_i)$, and let $\hat{Z}_i \subset \Pr^1_\C\times \Pr^1_\C$ by the graph of the action of $h_i=\mathsf{C A_i}$. We claim that the images $Z_1=\pi(\hat{Z}_1\cap (\Hp \times \Hp))$ and $Z_2=\pi(\hat{Z}_2 \cap (\Hp \times \Hp))$ are distinct; this implies that the real curves $C_1$ and $C_2$ in $Y$ with complex base change $Z_1$ and $Z_2$ respectively are distinct, hence $C_1(\R)\neq C_2(\R)$ (as these sets are infinite) and $\BA_Y$ is infinite.

To prove our claim, take $(x_2, y_2) \in \hat{Z}_2(\C)\cap (\Hp \times \Hp)$. If $\pi(x_2, y_2)=\pi(x_1, y_1)$ for some $(x_1, y_1)\in \hat{Z}_1(\C)$, then there exists $(\gamma, \gamma') \in \Gamma \times {^\tau \Gamma}$ such that $(x_1, y_1)$ belongs to the graph $\hat{Z}_{\tilde{h}_2}$ of the action of $\tilde{h}_{2}=\gamma'h_2\gamma$. Let us prove that $\hat{Z}_{\tilde{h}_2}\neq \hat{Z}_{h_1}$. Writing $\gamma'=\mathsf{C} \gamma''\mathsf{C}$ with $\gamma'' \in \Gamma$, if the two graphs were equal we would deduce that $\mathsf{A}_1$ and $\gamma''\mathsf{A}_2 \gamma \in \PGL(\R)$ are equal. But both elements belong to $B^\times/F^\times \subset \PGL(\R)$ and, as $L_1 \neq L_2$ and $\gamma$ and $\gamma''$ have norm one, the norms of any lifts of $\mathsf{A}_1$ and $\gamma''\mathsf{A}_2 \gamma$ to $B^\times$ do not differ by a square.

We have showed that $\hat{Z}_{\tilde{h}_2}\neq \hat{Z}_{h_1}$, hence the intersection of the two curves is finite. Varying $(\gamma, \gamma')$ we deduce that $Z_1 \cap Z_2$ is countable, therefore $Z_1 \neq Z_2$.

\begin{remark}
In fact, every bialgebraic curve in $Y(\R)$ arises via the construction in \S \ref{ssec:atr}, for some embedding of a non-CM quadratic extension $L/F$ in $B$. Indeed, as $\Gamma$ has finite index in $\tau(O^{\times, 1})$, the commensurators of $\Gamma$ and $\tau(O^{\times, 1})$ in $\PGL(\R)$ coincide, and it follows from \cite[Corollaire 1.5, p. 106]{vig80} that they are equal to $\tau(B^\times)$. Take $\mathsf{A} \in \PGL(\R)$ that is the image of a matrix with trace zero and negative determinant, and of the form $\mathsf{A}=\tau(\delta)$ for some $\delta \in B$. First of all $\delta$ cannot belong to $F$ (otherwise, having trace zero, it would be equal to zero), hence the $F$-algebra generated by $\delta$ is a quadratic extension $L$ of $F$ contained in $B$. Furthermore, the image of $\delta$ in $B \otimes_{F, \sigma_0}\R$ has negative determinant, hence $L\otimes_{F, \sigma_0}\R \simeq \R \times \R$.

On another note, let us remark that the geodesics in \S \ref{ssec:atr} are used in Greenberg's (resp. Guitart--Masdeu--Xarles's) generalization of Darmon's (resp. Darmon--Vonk's) construction of Stark--Heegner points (resp. $p$-adic singular moduli), cf. \cite[p. 560]{gre09} (resp. \cite[\S 4]{gmx21}).
\end{remark}

\subsubsection{Closed geodesics and bialgebraic curves} To conclude, for $\Gamma$ arithmetic and cocompact let us give a description of bialgebraic curves in $Y(\R)$ involving the Riemannian geometry of $Y(\R)$.

\begin{corollary}\label{cor:cocogeod}
Assume that $\Gamma$ is an arithmetic cocompact lattice. Every closed geodesic in $Y(\R)$ is contained in a bialgebraic curve; conversely, every bialgebraic curve in $Y(\R)$ contains a closed geodesic.
\end{corollary}
\begin{proof}
Closed geodesics in $Y(\R)=X(\C)$ are precisely images of geodesics $\mathcal{S}$ in $\Hp$ whose endpoints are the two fixed points of the action of an element $\gamma \in \Gamma \smallsetminus \{\mathrm{Id}\}$ (which is automatically an hyperbolic element of $\PGL(\R)$, as $\Gamma$ acts freely on $\Hp$ and $\Gamma \backslash \Hp$ is compact), cf. \cite[\S 6]{bb73}.

Take a closed geodesic $\mathcal{S}$ in $Y(\R)$ attached to an element $\gamma$ as above. Up to replacing $\gamma$ by a positive power, we may assume that it is the image of an element of $B^\times \smallsetminus F^\times$ (with the notation of \S \ref{ssec:atr}), which generates a quadratic non-CM extension $L/F$. Every element of $L^\times \smallsetminus F^\times$ has the same fixed points on $\Pr^1(\R)$ as $\gamma$; choosing an element whose trace in $F$ is zero we see, thanks to Lemma \ref{lem-geodbalg}, that $p(\mathcal{S})$ is contained in an irreducible real algebraic curve.

Conversely, let $\mathcal{C}=C(\R)\subset Y(\R)$ be a bialgebraic curve; we need to show that it contains a closed geodesic in $\Gamma \backslash \Hp$. We may replace $\Gamma$ by a subgroup of finite index and assume that we are in the situation of \S \ref{ssec:atr}. By \S \ref{ssec: bialisgeo} and the fact that $\mathrm{Comm}_{\PGL(\R)}(\Gamma)=\tau(B^\times)$, the set $\mathcal{C}$ contains the image of a geodesic $\mathcal{S}$ in $\Hp$ with equation $\mathsf{A}z=\bar{z}$, with $\mathsf{A}$ of trace zero and negative determinant, and of the form $\mathsf{A}=\tau(b)$ for some $b\in B^\times$. Let us show that the endpoints of $\mathcal{S}$ are the fixed points of some element of $\Gamma \smallsetminus \{\mathrm{Id}\}$. Let $L$ be the quadratic extension of $F$ generated by $b$, and let $O \subset B$ be an $\mathcal{O}_F$-order such that $\Gamma$ has finite index in $\tau(O^{\times, 1})$. The intersection $R=O \cap L$ is an $\mathcal{O}_F$-order in $L$, whose group of units has rank $1+\mathrm{rk}(\mathcal{O}_F^\times)$ \cite[Theorem I.12.12]{neu99}. Take a unit $\varepsilon \in R^\times$ of norm one in $F$, and generating an infinite subgroup with trivial intersection with $\mathcal{O}_F^\times$. Letting $m$ be the index of $\Gamma$ in $\tau(O^{\times, 1})$, we have $\tau(\varepsilon^m) \in \Gamma$, and the geodesic joining the fixed points of $\varepsilon^m$ is equal to $\mathcal{S}$.
\end{proof}

\begin{remark}
The above corollary implies that $\BA_Y$ consists exactly of Zariski closures of closed geodesics if and only if $\Gamma$ is arithmetic, mirroring the equivalence between points (1) and (2) in Theorem \ref{mainthmell}.
\end{remark}

\subsubsection{Proof of Theorem \ref{thm:altogether}}\label{ssec:altogether} The first point for $X$ the complement of a point in $\A^1_\C$ follows from \cite[Theorem 2.2.4]{Tam23}. If $X$ is compact of genus one, it follows from the first assertion in Theorem \ref{mainthmell} (and the fact that, for an étale cover $q: X' \rightarrow X$ as in the statement, the genus of $X'$ is one by the Riemann--Hurwitz formula). If $X$ is hyperbolic, then by Example \ref{ex-realbalg} the set of real points of $X'$ is a bialgebraic curve, whose image via (the Weil restriction of) the cover $X' \rightarrow X$ is contained in a bialgebraic curve in $Y(\R)$. Conversely, all bialgebraic curves in $Y(\R)$ come from covers as in the statement of the theorem by Proposition \ref{prop-bialfromrealcover}.

Let us now show the second point. If $X=\A^1_\C \smallsetminus \{0\}$ then the argument in the proof of \cite[Theorem 2.2.4]{Tam23} implies that bialgebraic curves in $Y(\R)$ are either lines through the origin or circles with center at the origin \footnote{See also \cite[Example 2.2.1]{Tam23}, but note that in the displayed formula $e^{-2\pi t}$ should be replaced by $e^{-4\pi t}$.}. In particular, there are two equivalence classes of bialgebraic curves in $Y(\R)$. If $X$ is compact of genus one, the statement follows from the equivalence between points (1) and (3) in Theorem \ref{mainthmell}. Finally, take $X$ hyperbolic; note that the automorphism group of $X$ is finite (an automorphism of $X$ extends to an automorphism of its compactification $\widehat{X}$, and there are finitely many automorphisms of $\widehat{X}$ fixing $\widehat{X} \smallsetminus X$). Therefore, the statement follows from Theorem \ref{thm-main}.

\bibliographystyle{amsalpha}
\bibliography{bialg}

\end{document}